%% file: main.tex
\documentclass[preprint,12pt]{elsarticle}

\usepackage{amsmath}
\usepackage{amssymb}
\usepackage{amsthm}
\usepackage{graphicx}
\usepackage{tcolorbox}
\usepackage{cancel}
\usepackage{ulem} 
\usepackage{multirow}
\usepackage{tabularx}
\usepackage{booktabs}
\usepackage[margin=1in]{geometry}
\usepackage{pifont} 
\usepackage{array,multirow}
\usepackage{makecell}
\usepackage{graphicx}
\usepackage{subcaption}
\usepackage{bm}
\usepackage{diagbox}
\usepackage{enumitem}

\newcolumntype{P}[1]{>{\centering\arraybackslash}p{#1}}

\DeclareMathOperator*{\dist}{dist}

\theoremstyle{plain}
\newtheorem{theorem}{Theorem}
\newtheorem{lemma}{Lemma}

\theoremstyle{definition}
\newtheorem{definition}{Definition}
\newtheorem{assumption}{Assumption}

\theoremstyle{remark}
\newtheorem{remark}{Remark}

\begin{document}
\begin{frontmatter}

\title{\textbf{Convergence Analysis via ODE Approach \\for Convex Optimization with Linear Equality Constraints}}

\author[aff1,aff3,aff4]{Chise Ishii\corref{cor1}}
\ead{chiseblanton125@keio.jp}

\author[aff2,aff3]{Yasushi Narushima}
\ead{narushima@keio.jp}

\cortext[cor1]{Corresponding author}

\affiliation[aff1]{organization={Graduate School of Science and Technology, Keio University},
            addressline={3-14-1 Hiyoshi, Kohoku-ku}, 
            city={Yokohama},
            postcode={223-8522}, 
            state={Kanagawa},
            country={Japan}}

\affiliation[aff2]{organization={Department of Industrial and Systems Engineering, Keio University},
            addressline={3-14-1 Hiyoshi, Kohoku-ku}, 
            city={Yokohama},
            postcode={223-8522}, 
            state={Kanagawa},
            country={Japan}}

\affiliation[aff3]{organization={Keio AI Research Center},
            addressline={4-1-1 Hiyoshi, Kohoku-ku}, 
            city={Yokohama},
            postcode={223-8521}, 
            state={Kanagawa},
            country={Japan}}

\affiliation[aff4]{organization={Mitsubishi UFJ Financial Group, Inc. and MUFG Bank,
Ltd.},
            addressline={1-4-5 Marunouchi, Chiyoda-ku}, 
            city={Tokyo},
            postcode={100-8388}, 
            country={Japan}}
 
\begin{abstract}
  This paper studies the continuous-time dynamics of primal-dual algorithms for linearly constrained convex optimization problems and provides a quantitative convergence analysis using the Lyapunov functions.
  With the growing prevalence of sparse and low-rank models, large-scale problems involving nonsmooth objective functions have become increasingly important.
  Our approach addresses nonsmooth and nonstrong convex objective functions, which is particularly effective in extending classical accelerated methods to broader large-scale optimization problems.
  Building upon the ordinary differential equation (ODE) approach inspired by the recent work on Nesterov's acceleration methods, we extend the analysis to an ODE associated with an optimization problem with linear equality constraints.
  Furthermore, to relax strong assumptions, we introduce two geometric conditions—the \textit{Flatness condition} and the \textit{Growth condition}—and derive explicit convergence rates depending on the local geometry and algorithmic parameter. 
  When the objective function is convex and the Lagrangian satisfies the Flatness condition, an appropriate parameter choice yields a worst-case convergence rate of $\mathcal{O}(1/t^2)$. 
  Furthermore, incorporating the Growth condition guarantees this worst-case rate of $\mathcal{O}(1/t^2)$ for any function satisfying these assumptions if and only if the parameter is chosen above a suitable threshold; under this condition, even faster convergence can be achieved.
  For the algorithmic construction, a numerical scheme is derived by discretizing the proposed ODE. 
  We also investigate the influence of algorithm parameters and provide insights into their optimal selection.
  Finally, preliminary numerical experiments are provided to validate the consistency with the theoretical results.
\end{abstract}


\end{frontmatter}

\input{chapter1}
\input{chapter2}
\input{chapter5_lem1}
\input{chapter5_lem2}

\input{chapter5_thm}
\input{chapter7}

\input{chapter6}

\input{ref}
\end{document}

%% file: chapter1.tex
\section{Introduction}
Consider the separable convex optimization problem with linear equality constraints:
\begin{align}\label{prob:main}
  \min_{\bm{x}} \left\{F(\bm{x}):=\sum_{i=1}^{m} f^i(x^i) \mid s.t.\ \sum_{i=1}^{m}A^ix^i= \bm{b},\ x^i\in \mathbb{R}^{n_i},\ i= 1,\ldots,m \right\}.
\end{align}
We define the matrix $A$ and the vector $\bm{x}$ as
\begin{align*}
N = \sum_{i=1}^m n_i,\quad
\bm{A} := \bigl[ A^{1},\dots, A^{m} \bigr]
\in \mathbb{R}^{r\times N},
\quad
\bm{x}:=
\begin{bmatrix}
x^{1} \\
\vdots \\
x^{m}
\end{bmatrix}\in \mathbb{R}^{N},
\end{align*}

where $m\in \mathbb{N}$ is an index that denotes the number of blocks, 
$\bm{A} \in \mathbb{R}^{r \times n_i}\ (rank(\bm{A}) = r)$ are linear operators and  $\bm{b} \in \mathbb{R}^r$ is a given vector. $f : \mathbb{R}^{n_i} \to \mathbb{R} \cup \{+\infty\}$ is a proper, lower semi-continuous, closed convex function, but not necessarily differentiable. We denote by $\operatorname{dom} F$ the domain of the function and $\overline{\operatorname{dom} F}$  the closure of $\operatorname{dom} F$. Let $\operatorname{dom} F\cap \{\bm{x}|\bm{A}\bm{x} = \bm{b}\}\neq \emptyset$.
 
We are mainly interested in first-order primal-dual methods for \eqref{prob:main}
based on the Lagrange function:
\begin{equation}\label{eq:lagrangian}
    \mathcal{L}(\bm{x},\bm{\lambda})
    := F(\bm{x})
       + \langle \bm{\lambda},\ \bm{A}\bm{x}- \bm{b}\rangle,
\end{equation}
where $(\bm{x},\bm{\lambda}) \in \mathbb{R}^{N} \times \mathbb{R}^r$ and $\mathcal{L}$ is proper, lower semi-continuous, and convex.
Throughout this paper, we assume that at least one saddle point $(\bm{x}^*, \bm{\lambda}^*) \in \mathcal{\bm{X}}^* \times \bm{\Lambda}^* \neq \emptyset$ exists, where $\mathcal{\bm{X}}^* \times \bm{\Lambda}^* := \arg\max_{\bm{\lambda} \ge \bm{0}} \arg\min_{\bm{x}} \mathcal{L}(\bm{x}, \bm{\lambda})$, and the optimal value is defined as $\mathcal{L}^* := \mathcal{L}(\bm{x}^*, \bm{\lambda}^*) = \min_{\bm{x}} \mathcal{L}(\bm{x}, \bm{\lambda}^*)$.
Then, we have
\begin{equation}\label{eq:saddle}
    \mathcal{L}(\bm{x}^{\ast}, \bm{\lambda})
    \ \le\ 
    \mathcal{L}(\bm{x}^{\ast}, \bm{\lambda}^\ast)
    \ \le\
    \mathcal{L}(\bm{x}, \bm{\lambda}^\ast),
    \qquad
    \forall (\bm{x},\bm{\lambda}) \in
    \mathbb{R}^{N}  \times \mathbb{R}^r.
\end{equation}
Note that linear inequality constraints can always be rewritten in the form 
$\bm{A}\bm{x}= \bm{b}$ by introducing nonnegative slack variables.  For example,
$\bm{A}\bm{x} \le \bm{b}$ is equivalent to $\bm{A}\bm{x} + \bm{s} = \bm{b}$ with $\bm{s} \ge 0$.\\
\indent Inspired by the pioneering work of Su et al.~\cite{su2014}, we study the following second-order ordinary differential inclusion (ODI):
\begin{align}\label{eq:ode}
  \begin{cases}
  0 \in 2 \ddot{\bm{x}}(t) + \frac{\alpha}{t}\dot{\bm{x}}(t)+ \partial_{\bm{x}} \mathcal{L}(\bm{x}(t),\bm{\lambda}(t))\\
  0 = \frac{\rho}{t}\dot{\bm{\lambda}}(t) - \nabla_{\bm{\lambda}}\mathcal{L}(\bm{x}(t) + \dot{\bm{x}}(t), \bm{\lambda}(t))
  \end{cases}
\end{align}
with $t_0>0$, an initial condition $\bm{x}(t_0) = \bm{x}_{0}$ and $\dot{\bm{x}}(t_0) = \bm{v}_0\in \operatorname{dom} F$.
Where $\nabla$ is a notion of the usual gradient and $\partial$ is a notion of generalized gradients so-called subdifferential and $\alpha, \rho$ are positive accelerate parameters.
Because $\mathcal{L}$ is proper, lower semi-continuous, and convex function, the subdifferential $\partial \mathcal{L}$ is a maximal monotone operator~\cite{rockafeller}, and therefore the Cauchy problem of \eqref{eq:ode} admits a unique global strong solution in the sense of Br\'ezis~\cite{brezis1973}\cite{brezis2011}.
 In particular, for every initial point $x_{0} \in \overline{\operatorname{dom} F}$, there exists a unique absolutely continuous trajectory $x(\cdot)$ satisfying \eqref{eq:ode} for a.e.\ $t>0$ and $\bm{x}(0) = \bm{x}_{0}$.

In this work, we investigate the decay of $F(x(t)) - F^*$ around saddle points by introducing geometric conditions on the Lagrangian function $\mathcal{L}$. 
We establish explicit convergence rates that depend directly on the local geometry and algorithmic parameter, achieving a worst-case convergence rate of $O(1/t^2)$ with optimal parameter choices.
Indeed, Attouch-et al.~\cite{attouch2013} showed the link between dynamical
stability and objective minimization using the combination of descent 
dynamics and the well-known geometric assumptions, \textit{Kurdyka- \text{\L}ojasiewicz (K\text{\L}) property}~\cite{kuadyka}\ (extended to nonsmooth functions by the work of Bolte et al.~\cite{bolte2007}), which characterizes the local geometry of functions around their critical points and plays an important role in nonsmooth, nonconvex analysis and modern convergence theory.

\subsection{Applications}
Although the model~\eqref{prob:main} is structurally simple, it is sufficiently expressive to cover a wide range of modern applications such as machine learning, signal processing, statistics, computer vision, among others.
support vector machines~\cite{74}, constrained sparse regression~\cite{18}, signal recovery~\cite{42, 40, 41, 70}, image restoration and denoising~\cite{72, 154, 134}.
For example, sparse signal recovery problems such as the lasso formulations can be written by introducing an auxiliary variable $x^1=x^2$, which immediately yields a two-block (namely, $m=2$) instance of \eqref{prob:main}.
Similarly, generalized LASSO models enforce sparsity on a transformed variable $Gx^1$ and become special cases of \eqref{prob:main} after introducing the constraint $Gx^1=x^2$.
In matrix completion, the rank-minimization model becomes convex when the rank is replaced by the nuclear norm, and a simple variable split again puts the problem into the same framework.
Similarly, robust principal component analysis (RPCA), which decomposes data into low-rank and sparse components, reduces—after replacing the rank by their convex surrogates and adding the constraint to a three-block instance of \eqref{prob:main}.
These representative examples highlight that many high-dimensional optimization models naturally possess separable objective functions and linear equality constraints, making \eqref{prob:main} a convenient and unified abstraction for large-scale convex optimization~\cite{de}.

\subsection{Related works}
\subsubsection{ADMM and its variants}
Although \textit{the alternating direction method of multipliers} (ADMM)~\cite{Boyd2011} originally proposed in the mid-1970s by Glowinski and Marrocco~\cite{glowinsky} and Gabay and Mercier~\cite{gabey} has been successfully applied in a wide range of modern applications, 
its theoretical guarantiees have historically required restrictive assumptions.
In particular, classical convergence analysis cannot obtain the accelerated rate of
$\mathcal{O}(1/k^{2})$ unless the objective function satisfies strong conditions such as strong convexity or $L$-smoothness.
Without such conditions, standard ADMM and its linearized variants generally achieve only an $\mathcal{O}(1/k)$ rate, failing to yield $\mathcal{O}(1/k^2)$ decay.

For example, Ouyang et al.~\cite{Ouyang} showed that the $\mathcal{O}(1/k^2)$ rate can be achieved by assuming $L$-smoothness, even without strong convexity. Conversely, Sabach and Teboulle~\cite{Sabach2022} established an $\mathcal{O}(1/k^2)$ rate under strong convexity within their unified FLAG framework, which encompasses various algorithms including ADMM.

Overall, these developments clearly indicate that attaining an $\mathcal{O}(1/k^2)$ convergence rate for ADMM has historically been possible only under such strong assumptions, leaving the relaxation of these requirements as a major theoretical barrier in prior analyses.

\subsubsection{Dynamical system approach}
In recent years, continuous-time dynamical systems that arise as limits of optimization algorithms have attracted considerable attention in an unconstrained setting. 
This dynamical systems approach to convergence analysis, pioneered by Su et al.~\cite{su2014}, leverages the analogy between continuous optimization and numerical analysis. 
They derived a second-order ordinary differential equation as the continuous-time limit of Nesterov's accelerated gradient method (NAG) and demonstrated equivalence between the discrete scheme and NAG, 
establishing that continuous models can effectively analyze optimization algorithms.

This framework offers several advantages over traditional optimization-theoretic analyses. 
First, the continuous limit allows chain rules to be employed, simplifying analysis and potentially enabling results without strong assumptions (e.g., the L-smoothness). 
Second, the dynamics formulation enables subdifferential systems and extensions to discontinuous and nonconvex settings. 
Third, ideas from numerical discretization can directly relate to dynamical properties, clarifying how parameters such as step sizes affect convergence rates. 
Finally, careful design of the ordinary differential equations (ODE) suggests that appropriate discretization schemes may yield high-quality algorithms.

Despite these advantages, a unified framework for this analysis technique has remained underdeveloped.
Recently, Ushiyama et al.~\cite{Ushiyama} established a correspondence between certain discrete algorithms and their continuous ODE models by introducing the concept of the “\textit{weak discrete gradient (wDG)}”.
Their framework provides a systematic bridge between continuous and discrete regimes for a class of gradient-based flows in unconstrained optimization.
When applicable, continuous-time Lyapunov analyses can be translated into discrete-time convergence guarantees within this \textit{wDG} framework.
Therefore, analyzing continuous ODE models provides a valuable theoretical foundation for designing and analyzing corresponding discrete optimization algorithms.

Regarding the geometric assumptions, Aujol et al.~\cite{aujol2019}, assuming that the objective function is differentiable, analyzed the second-order ODE arising as the continuous limit of FISTA using the Lyapunov function, 
deriving convergence rates depending on the local geometry and algorithmic parameter, and establishing the classical $\mathcal{O}(1/t^2)$ local rate under a certain condition.
To handle nonsmoothness, Luo~\cite{moreau} analyzed the second-order
ordinary differential inclusion (ODI) employing  the Moreau-Yosida regularization, which allows to establish both the well-posedness of the continuous dynamics and a rigorous discretization into proximal-type algorithms. 

While the dynamical systems approach has been extensively studied for unconstrained optimization problems,  
Luo and Zhang~\cite{Luo2023} and Luo~\cite{Luo2026} applied it to constrained optimization problems, proposing the accelerated primal-dual (APD) flow.
While Luo and Zhang required strong convexity of at least one of the objective terms to establish linear convergence, Luo recently proved that linear convergence can be achieved under mere convexity by replacing the standard Lagrangian with an augmented Lagrangian.
After establishing linear convergence at the continuous level, they performed time discretization and analyzed the resulting scheme via a discrete Lyapunov function, 
obtaining the ADMM-type algorithm with nonergodic mixed-rate convergence guarantees. 
They identified extension to nonsmooth settings at the continuous level as future work, which remains an important direction. These results are summarized in Table~\ref{tab:1}.

\newcolumntype{C}[1]{>{\centering\arraybackslash}p{#1}}

\begin{table}[htbp]
  \centering
  \caption{Relationship between the Objective Function and the Associated Dynamics}
  \label{tab:1}
  \scriptsize
  \setlength{\tabcolsep}{1.5pt} 
  \renewcommand{\arraystretch}{1.3} 

  \begin{tabular}{l C{1.8cm} c C{1.2cm} C{1.8cm} ccc ccc C{1.2cm} C{1.2cm}}
    \toprule
    & & & & & \multicolumn{3}{c}{Convexity} & \multicolumn{3}{c}{Continuity} & & \\
    \cmidrule(lr){6-8} \cmidrule(lr){9-11}
    Type & Dynamics & Order & Rate & Study & \makecell{Strong\\conv.} & Conv. & \makecell{Non-\\conv.} & \makecell{Diff.} & \makecell{Loc.\\Lips.} & Lsc & \makecell{Geom.\\assump.} & \makecell{Moreau\\approx.} \\
    \midrule

    \multirow{3}{*}{\makecell[l]{Non-\\constraint}}
      & NAG flow & 2nd ODE & $O(1/t^{2})$ & Su et al. \cite{su2014}
      & -- & $f$ & -- & $\nabla f$ & -- & -- & -- & -- \\
      & NAG flow & 2nd ODI & $O(1/t^{2})$ & Luo \cite{moreau}
      & -- & -- & $f$ & -- & -- & $f$ & -- & $f$ \\
      & Polyak's ODE & 2nd ODE & $O(1/t^{2})$ & Aujol et al. \cite{aujol2019}
      & -- & $f$ & -- & $f$ & -- & -- & $f$ & -- \\

    \midrule 

    \multirow{2}{*}{\makecell[l]{Constraint}}
      & APD flow & 2nd ODI & $O(e^{-t})$ & Luo and Zhang \cite{Luo2023}
      &  $f$ (or $g$) & $g$ (or $f$) & -- & $\nabla f, \nabla g$ & -- & -- & -- & -- \\
      & APD flow & 2nd ODI & $O(e^{-t})$ & Luo \cite{Luo2026}
      & -- & $f, g$  & -- & $\nabla f, \nabla g$ & -- & -- & -- & -- \\
    \bottomrule
  \end{tabular}
\end{table}

\subsection{Motivations and contributions}
In many large-scale and nonsmooth optimization problems, the strong convexity and the smoothness assumptions often do not hold.
Nevertheless, existing studies on convex optimization problems with linear equality constraints have generally required strong assumptions such as strong convexity and smoothness, to establish the conventional convergence rates of $\mathcal{O}(1/t^{2})$.
Our goal is to develop algorithms that relax the strong assumptions requirement and extend classical accelerated methods to broader problem classes.
To address this, we propose a novel dynamical system derived as a continuous limit of the primal-dual algorithms for convex optimization problems with linear equality constraints and perform convergence analysis without such strong assumptions.

In this study, we conduct a quantitative Lyapunov-based analysis for the Lagrangian function under nonsmooth and non-strongly convex settings for the objective function.
 To relax strong assumptions, we introduce two geometric conditions—the \textit{Flatness condition} and the \textit{Growth condition}—to establish explicit local convergence rates that depend on the local geometry and algorithmic parameter. 
 Specifically, when the objective function is convex and the Lagrangian function satisfies the \textit{Flatness condition}, optimal parameter selection recovers a worst-case local convergence rate of $O(1/t^2)$. 
 Furthermore, by incorporating the \textit{Growth condition}, we not only guarantee this $O(1/t^2)$ worst-case bound across all qualifying functions, but also attain accelerated convergence rates.
In addition, we derive discrete algorithms via appropriate numerical discretization schemes and verify the consistency with the theoretical results through a series of numerical experiments.

\subsection{Organization of the paper}
The paper is organized as follows.
Section~2 introduces the geometric assumptions imposed on $F$ in unconstrained optimization problems and extends to linearly constrained optimization problems.
Section~3 presents our main results, which consist of a convergence analysis for the continuous-time
model. Section~4 is devoted to discretization of the dynamics and to numerical experiments
that illustrate the theoretical findings.

%% file: chapter2.tex
\section{Local geometry of convex functions}
 To analyze the local behavior of the Lagrange function around its saddle points, we introduce two geometric properties: the 
 \textit{Flatness condition} and the \textit{Growth condition}. In this section, we first present these conditions for unconstrained optimization problems and then extend them to linearly constrained problems, which is instrumental in proving our main result.
 
\subsection{Flatness and growth conditions for unconstrained optimization problems}
First, we consider the unconstrained optimization problem
\begin{align}\label{eq:unconstrained}
\min_{\bm{x}\in \mathbb{R}^N} F(\bm{x}),
\end{align}
where $F:\mathbb{R}^N\to \mathbb{R}\cup\{+\infty\}$ is a proper, lower semicontinuous, convex, not necessarily differentiable function. 
For any $\bm{x} \in \mathbb{R}^N$, we define the distance to $\Omega$ by$\mathrm{dist}(\bm{x},\Omega) := \inf_{\bm{\zeta}\in \Omega}\|\bm{x}-\bm{\zeta}\|,$
where $\|\cdot\|$ denotes the Euclidean norm. $B({\bm{x}^*}, \varepsilon)$ denotes the open Euclidean ball with center ${\bm{x}^*}$ and radius $\varepsilon > 0$.
Since $F$ is a convex function  and $\bm{x}\in \operatorname{int} \operatorname{dom}F$, the subdifferential $\partial F(\bm{x})$ is nonempty, thus we take an arbitrary element $\zeta \in \partial F(\bm{x})$.
Here, we introduce the \textit{Flatness condition} $\mathbf{H_0 1}(\gamma)$ and \textit{Growth condition} $\mathbf{H_0 2}(r)$ for the objective function $F$. 
The following definition is taken from~\cite{aujol2019}.

\begin{definition}[Flatness and Growth conditions] \label{def:1}
  Let \( F :\mathbb{R}^{N} \to  \mathbb{R} \cup \{+\infty\}\) be a proper, lower semicontinuous, convex, not necessarily differentiable function. 
  We denote the set of minimizers of $F$ by $\mathcal{X}^* := \arg\min_{\bm{x}} F(\bm{x})\neq\emptyset$ and the optimal value by $F^* := \inf F(\bm{x})$. 
Let $r_0>0$, $\varepsilon>0$, $K>0$ and take an arbitrary element $\zeta \in \partial F(\bm{x})$.
  \begin{itemize}
  \item[(i)] [Flatness condition \( \mathbf{H_0 1}(\gamma) \)]
  The function \( F \) satisfies hypothesis \( \mathbf{H_0 1}(\gamma) \) with $\gamma \geq 1$ if, for any minimizer \( \bm{x}^{*} \in \mathcal{X}^* \), there exists a positive constant \( \varepsilon \) such that:
  \begin{equation}\label{eq:h1}
    F(\bm{x}) - F^* \leq \frac{1}{\gamma} \langle \zeta, \bm{x} - \bm{x}^{*} \rangle
  \end{equation}
  for all $\bm{x} \in B(\bm{x}^*, \varepsilon) \cap \{ \bm{x} \mid F(\bm{x}^*) < F(\bm{x}) < F(\bm{x}^*) + r_0 \}$.

\item[(ii)] [Growth condition \( \mathbf{H_0 2}(r) \)]
  The function \( F \) satisfies hypothesis \( \mathbf{H_0 2}(r) \) with $r \geq 1$ if, for any minimizer \( \bm{x}^{*} \in \mathcal{X}^* \), there exist positive constants \( \varepsilon \) and \( K \) such that:
  \begin{equation}\label{eq:h2}
    K \cdot \operatorname{dist}(\bm{x}, \mathcal{X}^*)^r \leq F(\bm{x}) - F^*
  \end{equation}
  for all $\bm{x} \in B(\bm{x}^*, \varepsilon) \cap \{ \bm{x} \mid F(\bm{x}^*) < F(\bm{x}) < F(\bm{x}^*) + r_0 \}$.
  \end{itemize}
\end{definition}
$\mathbf{H_0 1}(1)$ holds for any convex function, and $\mathbf{H_0 1}(\gamma)$ is a mild assumption, requesting slightly more than the convexity of $F$ in the neighborhood
of its minimizer.
Note that when $F$ is convex, the \textit{Growth condition} $\mathbf{H_0 2}(r)$ corresponds to the well-known \textit{\text{\L}ojasiewicz gradient inequality}~$c(F(\bm{x}) - F({\bm{x}^*}))^{\eta}\le \dist(\bm{0}, \zeta)$~\cite{loja} with exponent \( \eta = 1 - \frac{1}{r} \).
The constant \( \eta \) is called the \textit{\text{\L}ojasiewicz exponent}, and when \( \eta = \frac{1}{2} \), it corresponds to the quadratic growth condition\ (see~\cite{bolte2016} Section 3, \cite{aujol2019} Section 2),
hence any strongly convex function satisfies $\mathbf{H_0 2}(2)$.

\subsection{Flatness and growth conditions for constrained optimization problems}
We extend \( \mathbf{H_0 1}(\gamma) \) and \( \mathbf{H_0 2}(r) \) to the linearly constrained optimization problem~\eqref{prob:main}.
Instead of the above conditions, we define \textit{Flatness condition} \( \mathbf{H_1}(\gamma) \) and \textit{Growth condition} \( \mathbf{H_2}(r) \) in terms of the assumptions on the Lagrangian function.
\begin{definition}[Flatness and Growth conditions for the Lagrangian function] Let \( \mathcal{L} :\mathbb{R}^{N} \to  \mathbb{R} \cup \{+\infty\}\) be a proper, lower semicontinuous, convex, not necessarily differentiable function. 
  We denote the set of saddle points of $\mathcal{L}$ by
    $(\mathcal{X}^*,\Lambda^*) \in \arg\min_{\bm{x}} \mathcal{L}(\bm{x}, \bm{\lambda}^*) \cap \arg\max_{\bm{\lambda}} \mathcal{L}(\bm{x}^*, \bm{\lambda}) \neq \emptyset$ and the optimal value as $\mathcal{L}(\bm{x}^*, \bm{\lambda}^*) := \min_{\bm{x}}\max_{\bm{\lambda}} \mathcal{L}(\bm{x}, \bm{\lambda})$. 
  Let $r_0>0$, $\varepsilon>0$, $K'>0$ and take an arbitrary element $\zeta \in \partial F(\bm{x})$.
  \begin{itemize}
  \item[(i)] [Flatness condition for the Lagrangian function \( \mathbf{H1}(\gamma) \)]
  The Lagrangian function \( \mathcal{L} \) satisfies hypothesis \( \mathbf{H1}(\gamma) \) with $\gamma \geq 1$ if, for any saddle point \( (\bm{x}^{*}, \bm{\lambda}^*) \in (\mathcal{X}^*, \bm{\Lambda}^*) \), there exists a positive constant \( \varepsilon \) such that:
\begin{align}\label{lemma:flatness}
 \mathcal{L}(\bm{x}, \bm{\lambda}^*) - \mathcal{L}(\bm{x}^*, \bm{\lambda}) &\leq 
\frac{1}{\gamma}
\begin{pmatrix}
  \zeta + \bm{A}^{\top}\bm{\lambda} \\
  -\nabla_{\bm{\lambda}} \mathcal{L}(\bm{x}, \bm{\lambda})
\end{pmatrix}^{\top}
\begin{pmatrix}
  \bm{x} - \bm{x}^* \\
  \bm{\lambda} - \bm{\lambda}^*
\end{pmatrix}.
  \end{align}
  for all $\begin{pmatrix}
      \bm{x} \\
      \bm{\lambda}
    \end{pmatrix} \in 
    B\Bigl(\begin{pmatrix}
      \bm{x}^* \\
      \bm{\lambda}^*
    \end{pmatrix}, \varepsilon\Bigr)\cap 
    \Bigl\{ \begin{pmatrix}
      \bm{x} \\
      \bm{\lambda}
  \end{pmatrix} \mid \mathcal{L}(\bm{x}^*, \bm{\lambda}^*) - r_0 \le \mathcal{L}(\bm{x}^*, \bm{\lambda}) \le \mathcal{L}(\bm{x}, \bm{\lambda}^*) \le \mathcal{L}(\bm{x}^*, \bm{\lambda}^*) + r_0 \Bigr\}.$

\item[(ii)] [Growth condition for the Lagrangian function \( \mathbf{H2}(r) \)]
The Lagrangian function $\mathcal{L}$ satisfies hypothesis \textbf{H2}$(r)$ with $1\le r\le 2$ if, for any saddle point \( (\bm{x}^{*}, \bm{\lambda}^*) \in (\mathcal{X}^*, \bm{\Lambda}^*) \), there exists a positive constant \( \varepsilon, K \) such that:
\begin{align}\label{lemma:growth}
 K \cdot \operatorname{dist}\Bigl(
      \begin{pmatrix}
      \bm{x} \\
      \bm{\lambda}
    \end{pmatrix},
    \begin{pmatrix}
      \bm{\mathcal{X}}^* \\
      \bm{\Lambda}^*
    \end{pmatrix}\Bigr)^r
    \leq
    \mathcal{L}(\bm{x},\bm{\lambda}^*) - \mathcal{L}(\bm{x}^{*}, \bm{\lambda})
\end{align}
for all $\begin{pmatrix}
      \bm{x} \\
      \bm{\lambda}
    \end{pmatrix} \in 
    B\Bigl(\begin{pmatrix}
      \bm{x}^* \\
      \bm{\lambda}^*
    \end{pmatrix}, \varepsilon\Bigr)\cap \Bigl\{ 
      \begin{pmatrix}
      \bm{x} \\
      \bm{\lambda}
    \end{pmatrix}
     \mid \mathcal{L}(\bm{x}^*, \bm{\lambda}^*) - r_0 \le \mathcal{L}(\bm{x}^*, \bm{\lambda}) \le \mathcal{L}(\bm{x}, \bm{\lambda}^*) \le \mathcal{L}(\bm{x}^*, \bm{\lambda}^*) + r_0 \Bigr\}.$
\end{itemize}
\end{definition}

\begin{remark}
  If $F$ satisfies $F(\bm{x}) - F(\bm{y}) \leq \frac{1}{\gamma} \langle \zeta, \bm{x} - \bm{y} \rangle$
  for all $\bm{x}, \bm{y}$, then the Lagrangian function satisfies Hypothesis $\mathbf{H1}(\gamma)$.
 We can see that
\begin{align*}
  \mathcal{L}(\bm{x},\bm{\lambda}^*) - \mathcal{L}(\bm{x}^{*}, \bm{\lambda})
  &=
\bigl(F(\bm{x}) +\langle \bm{\lambda}^*,  \bm{A} \bm{x} - \bm{b}\rangle\bigr)
        -\bigl(F^* +\langle \bm{\lambda},  \bm{A} \bm{x}^{*} - \bm{b} \rangle\bigr)\\
  &= F(\bm{x})-F^* + \langle \bm{\lambda}^*,  \bm{A} \bm{x} - \bm{b}\rangle\\
  &\leq \frac{1}{\gamma}\langle\zeta,  \bm{x} - \bm{x}^{*}\rangle
  +\frac{1}{\gamma}\langle \bm{A}^{\top}\bm{\lambda},  \bm{x} - \bm{x}^{*}\rangle
  - 
     \frac{1}{\gamma}\langle \bm{\lambda},  \bm{A} \bm{x} -  \bm{A} \bm{x}^{*}
  \rangle
+ \langle \bm{\lambda}^*,  \bm{A} \bm{x} - \bm{b}\rangle\\\notag
  &= \frac{1}{\gamma}\Bigl(\langle\zeta + \bm{A}^{\top}\bm{\lambda}, \bm{x} - \bm{x}^{*}\rangle - \langle \bm{A} \bm{x} - \bm{b}, \bm{\lambda} - \bm{\lambda}^*\rangle\Bigr).
   \end{align*}
Especially, if $F$ is convex, the Lagrangian function satisfies $\mathbf{H1}(1)$.
\end{remark}

%% file: chapter5_lem1.tex
\section{Continuous dynamical systems}
In this section, we state the optimal convergence rates that can be achieved.
We consider the following second-order ordinary differential inclusion (ODI), called the \textit{Accelerated Primal-Dual Flow}:
\begin{align}\label{eq:dynamics}
  \begin{cases}
  0 \in 2 \ddot{\bm{x}}(t) + \frac{\alpha}{t}\dot{\bm{x}}(t)+ \partial_{\bm{x}} \mathcal{L}(\bm{x}(t),\bm{\lambda}(t))\\
  0 = \frac{\rho}{t}\dot{\bm{\lambda}}(t) - \nabla_{\bm{\lambda}}\mathcal{L}(\bm{x}(t) + \dot{\bm{x}}(t), \bm{\lambda}(t))
  \end{cases}
\end{align}
with $t_0 > 0, \bm{x}(t_0) = \bm{x}_0$ and $\dot{\bm{x}}(t_0) = \bm{v}_0$.
We equip a Lyapunov function $\mathcal{E}(t)$:
\begin{align}\label{eq:Lyapunov}
  \mathcal{E}(t) &=\frac{t^2}{2}\left( \mathcal{L}(\bm{x}(t), \bm{\lambda}^*) - \mathcal{L}(\bm{x}^*, \bm{\lambda}(t)) \right)\notag\\
  &\quad + \frac{1}{2} \left\| s(\bm{x}(t) - \bm{x}^*)+ t\dot{\bm{x}}(t)\right\|^2 
  + \frac{\xi}{2} \left\| \bm{x}(t) - \bm{x}^* \right\|^2
  + \theta \| \bm{\lambda}(t) - \bm{\lambda}^* \|^2
\end{align}
where $(\bm{x}^*, \bm{\lambda}^*)$ is a saddle point of $\mathcal{L}$ and $s, \xi, \theta$ are three real  constants.
$\mathcal{E}$ is nonnegative, vanishes only at the equilibrium point $(\bm{x}^*, \bm{\lambda}^*)$, and decreases along trajectories of the ODI~\eqref{eq:dynamics}.
\begin{lemma}[Upper bound for the Lyapunov function]\label{lem:Lyapunov}
 Assume that \textbf{H1}$(\gamma)$ with $\gamma \geq 1$. For sufficiently large $t$, we obtain an upper bound for the Lyapunov function by choosing $\xi = s(s - \frac{\alpha}{2} + 1)$.
    \begin{align}
 \mathcal{E}'(t)    
&\leq t\Bigl(1 - \frac{s\gamma}{2}\Bigr)\{ \mathcal{L}(\bm{x},\bm{\lambda}^*) - \mathcal{L}(\bm{x}^*,\bm{\lambda})\}
 +  \frac{1}{t}\Bigl(s -\frac{\alpha}{2} + 1\Bigr)\left\|s (\bm{x} - \bm{x}^{*}) + t\dot{\bm{x}}\right\|^2\notag\\
&\quad+ \frac{1}{t}\left\{\Big(\frac{2\theta}{\rho}- \frac{s\gamma}{2}\Big)\cdot\frac{\kappa^2}{2} -s^2\Big(s-\frac{\alpha}{2} + 1\Big)\right\}\left\| \bm{x} - \bm{x}^{*}\right\|^2.
\end{align}
\end{lemma}
\begin{proof}
Note that we take an arbitrary element $\zeta \in \partial F(\bm{x})$ and define $\partial\mathcal{L}(\bm{x}, \bm{\lambda}) = \partial F(\bm{x}) + \bm{A}^{\top}\bm{\lambda}:= \{\zeta + \bm{A}^{\top}\bm{\lambda}|\zeta\in\partial F(\bm{x})\}$. 
From now on, we denote $(\bm{x}, \bm{\lambda})$ by a time-dependent trajectory $(\bm{x}(t), \bm{\lambda}(t))$. Taking the time derivative of $\mathcal{E}(t)$ defined in~\eqref{eq:Lyapunov} and applying the ODI~\eqref{eq:dynamics} yields the following result.
 \begin{align}\label{eq:19}
    \mathcal{E'}(t) 
    &= t \{\mathcal{L}(\bm{x}, \bm{\lambda}^*) - \mathcal{L}(\bm{x}^*, \bm{\lambda}) \}
     + \frac{t^2}{2} \left\langle \zeta + \bm{A}^{\top}\bm{\lambda}^*, \dot{\bm{x}} \right\rangle 
     - \frac{t^2}{2} \left\langle \underbrace{\nabla_{\bm{\lambda}}\mathcal{L}(\bm{x}^*, \bm{\lambda})}_{\bm{A}\bm{x}^* - b = 0},\dot{\bm{\lambda}} \right\rangle\notag \\
    &\quad + \left\langle s\dot{\bm{x}}+ t \ddot{\bm{x}} + \dot{\bm{x}}, s( \bm{x} - \bm{x}^{*}) + t\dot{\bm{x}} \right\rangle 
 + \xi \left\langle \dot{\bm{x}},  \bm{x} - \bm{x}^{*} \right\rangle +2\theta\langle\dot{\bm{\lambda}}, \bm{\lambda} - \bm{\lambda}^* \rangle\notag\\
    &= t \{\mathcal{L}(\bm{x}, \bm{\lambda}^*) - \mathcal{L}(\bm{x}^*, \bm{\lambda}) \}\notag\\
      &\quad + \frac{t^2}{2} \left\langle {\zeta + \bm{A}^{\top}\bm{\lambda}^*+{2\ddot{\bm{x}}}},\ \dot{\bm{x}} \right\rangle 
      + (s + 1)t \| \dot{\bm{x}} \|^2 + s t \langle {{\ddot{\bm{x}}}},  \bm{x} - \bm{x}^{*}  \rangle 
      + (s(s + 1) + \xi) \langle \dot{\bm{x}},  \bm{x} - \bm{x}^{*} \rangle \notag\\
      &\quad  +2 \theta\langle{{\dot{\bm{\lambda}}}}, \bm{\lambda} - \bm{\lambda}^* \rangle \notag\\
      &= t \{\mathcal{L}(\bm{x}, \bm{\lambda}^*) - \mathcal{L}(\bm{x}^*, \bm{\lambda})\}\notag\\
    &\quad +\frac{t^2}{2} \left\langle {(\zeta + \bm{A}^{\top}\bm{\lambda}^*){-(\zeta + \bm{A}^{\top}\bm{\lambda})-\frac{\alpha}{t}\dot{\bm{x}}}}, \dot{\bm{x}} \right\rangle
      + (s + 1)t \| \dot{\bm{x}} \|^2 + \frac{st}{2} \langle {{-(\zeta + \bm{A}^{\top}\bm{\lambda})-\frac{\alpha}{t}\dot{\bm{x}}}},  \bm{x} - \bm{x}^{*} \rangle\notag\\
      &\quad +  (s(s + 1) + \xi) \langle \dot{\bm{x}},  \bm{x} - \bm{x}^{*} \rangle
     +2\theta\left\langle{{{\frac{t\nabla_{\bm{\lambda}}\mathcal{L}(\bm{x}+\dot{\bm{x}}, \bm{\lambda})}{\rho}} }}, \bm{\lambda} - \bm{\lambda}^* \right\rangle\notag\\
    &= t \{\mathcal{L}(\bm{x}, \bm{\lambda}^*) - \mathcal{L}(\bm{x}^*, \bm{\lambda})\} \notag\\
    &\quad +\frac{t^2}{2} \left\langle \bm{A}^{\top}(\bm{\lambda}^*-\bm{\lambda}) -\frac{\alpha}{t}\dot{\bm{x}}, \dot{\bm{x}} \right\rangle
      + (s + 1)t \| \dot{\bm{x}} \|^2 - \frac{st}{2} \langle {\zeta + \bm{A}^{\top}\bm{\lambda}+\frac{\alpha}{t}\dot{\bm{x}}},  \bm{x} - \bm{x}^{*} \rangle \notag\\
      &\quad+  (s(s + 1) + \xi) \langle \dot{\bm{x}},  \bm{x} - \bm{x}^{*} \rangle 
       +\frac{2t\theta}{\rho}\langle(\bm{A}(\bm{x} + \dot{\bm{x}})-b), \bm{\lambda} - \bm{\lambda}^* \rangle\notag\\
&= t \{\mathcal{L}(\bm{x}, \bm{\lambda}^*) - \mathcal{L}(\bm{x}^*, \bm{\lambda})\}\notag\\
&\quad
+ t\Big(s + 1 - \frac{\alpha}{2}\Big)\left\| \dot{\bm{x}} \right\|^2
- \frac{st}{2} \left\langle 
\zeta+ \bm{A}^{\top}\bm{\lambda}, 
 \bm{x} - \bm{x}^{*} 
\right\rangle 
+ \Big(s(s + 1) + \xi - \frac{s\alpha}{2}\Big) 
\left\langle 
\dot{\bm{x}}, 
 \bm{x} - \bm{x}^{*} 
\right\rangle
\notag\\
&\quad +
\frac{t^2}{2}\left\langle
\bm{A}\dot{\bm{x}} 
,
\bm{\lambda}^* - \bm{\lambda}
\right\rangle
 + \frac{2t\theta}{\rho} \left\langle 
\bm{A}
( \bm{x} - \bm{x}^{*} )
+
\bm{A}
\dot{\bm{x}} 
,
(\bm{\lambda} - \bm{\lambda}^*)
\right\rangle \notag\\
&= t \{\mathcal{L}(\bm{x}, \bm{\lambda}^*) - \mathcal{L}(\bm{x}^*, \bm{\lambda})\}\notag\\
&\quad
+ t\Big(s + 1 - \frac{\alpha}{2}\Big) 
\left\| \dot{\bm{x}} \right\|^2
- \frac{st}{2} \left\langle 
\zeta+ \bm{A}^{\top}\bm{\lambda}, 
 \bm{x} - \bm{x}^{*} 
\right\rangle 
+ \Big(s(s + 1) + \xi - \frac{s\alpha}{2}\Big) 
\left\langle 
\dot{\bm{x}}, 
 \bm{x} - \bm{x}^{*} 
\right\rangle
\notag\\
&\quad 
+ \frac{2t\theta}{\rho} \left\langle 
\bm{A}\bm{x} - \bm{b}
,
\bm{\lambda} - \bm{\lambda}^*
\right\rangle
+
\Big(\frac{2t\theta}{\rho}-\frac{t^2}{2}\Big)
\left\langle
  \bm{A}
\dot{\bm{x}} 
,
\bm{\lambda} - \bm{\lambda}^*
\right\rangle.
\end{align}
Applying ~\eqref{lemma:flatness} to the third term in \eqref{eq:19}, the following relation holds:
\begin{align}
&- \frac{st}{2} \left\langle 
\zeta+ \bm{A}^{\top}\bm{\lambda}, 
 \bm{x} - \bm{x}^{*} 
\right\rangle 
    + \frac{2t\theta}{\rho} \left\langle 
\bm{A}\bm{x} - \bm{b}
,
\bm{\lambda} - \bm{\lambda}^*
\right\rangle\notag\\
&\qquad\quad\leq
- \frac{st}{2}\Bigl(\gamma\{\mathcal{L}(\bm{x}, \bm{\lambda}^*) - \mathcal{L}(\bm{x}^*, \bm{\lambda})\} + \langle \bm{A}\bm{x} -\bm{b}, \bm{\lambda} - \bm{\lambda}^*\rangle\Bigr)
+ \frac{2t\theta}{\rho} \left\langle 
\bm{A}\bm{x} - \bm{b}
,
\bm{\lambda} - \bm{\lambda}^*
\right\rangle\notag\\
&\qquad\quad= - \frac{st\gamma}{2}\{\mathcal{L}(\bm{x}, \bm{\lambda}^*) - \mathcal{L}(\bm{x}^*, \bm{\lambda})\} + \Bigl(\frac{2t\theta}{\rho} - \frac{st}{2}\Bigr) \left\langle 
\bm{A}\bm{x} - \bm{b}
,
\bm{\lambda} - \bm{\lambda}^*
\right\rangle\notag\\
&\qquad\quad= - \frac{st\gamma}{2}\{\mathcal{L}(\bm{x}, \bm{\lambda}^*) - \mathcal{L}(\bm{x}^*, \bm{\lambda})\} + \Bigl(\frac{2t\theta}{\rho} - \frac{st}{2}\Bigr) \left\langle 
\bm{A}(\bm{x} - \bm{x}^*)
,
\bm{\lambda} - \bm{\lambda}^*
\right\rangle.
\end{align}
Then, \eqref{eq:19} can be written by
\begin{align}\label{eq:21}
\mathcal{E'}(t)
&\leq
\Bigl(t - \frac{st\gamma}{2}\Bigr)\{\mathcal{L}(\bm{x}, \bm{\lambda}^*) - \mathcal{L}(\bm{x}^*, \bm{\lambda})\}\notag\\
&\quad
+  t\Big(s + 1 - \frac{\alpha}{2}\Big) 
\left\| \dot{\bm{x}} \right\|^2
+ \Big(s(s + 1) + \xi - \frac{s\alpha}{2}\Big) 
\left\langle 
\dot{\bm{x}}, 
 \bm{x} - \bm{x}^{*} 
\right\rangle
\notag\\
&\quad 
 + \Bigl(\frac{2t\theta}{\rho} - \frac{st}{2}\Bigr) \left\langle 
\bm{A}(\bm{x} - \bm{x}^*)
,
\bm{\lambda} - \bm{\lambda}^*
\right\rangle
+
\Big(\frac{2t\theta}{\rho}-\frac{t^2}{2}\Big)
\left\langle
  \bm{A}
\dot{\bm{x}} 
,
\bm{\lambda} - \bm{\lambda}^*
\right\rangle.
\end{align}
We apply Young's inequality to the last two inner product terms in \eqref{eq:21} and get
\begin{align}
    &\quad\Bigl(\frac{2t\theta}{\rho} - \frac{st}{2}\Bigr) \left\langle 
\bm{A}(\bm{x} - \bm{x}^*)
,
\bm{\lambda} - \bm{\lambda}^*
\right\rangle
+
\Big(\frac{2t\theta}{\rho}-\frac{t^2}{2}\Big)
\left\langle
  \bm{A}
\dot{\bm{x}} 
,
\bm{\lambda} - \bm{\lambda}^*
\right\rangle\notag\\
&\leq
 \Big(\frac{2t\theta}{\rho}- \frac{st}{2}\Big)\Big[\frac{1}{2}
\left\|\frac{1}{t}\bm{A}(
 \bm{x} - \bm{x}^{*})\right\|^2
+\frac{1}{2}\left\|t
(\bm{\lambda} - \bm{\lambda}^*)\right\|^2
\Big]
+
\Big(\frac{2t\theta}{\rho}-\frac{t^2}{2}\Big)
\Big[\frac{1}{2}\left\|\frac{1}{t}
\bm{A}
\dot{\bm{x}} 
\right\|^2
+\frac{1}{2}\left\|t(\bm{\lambda} - \bm{\lambda}^*)\right\|^2
\Big]\notag\\
&
=\Big(\frac{2t\theta}{\rho}- \frac{st}{2}\Big)\cdot\frac{\kappa^2}{2t^2}
\left\|
 \bm{x} - \bm{x}^{*}\right\|^2 
+\left\{\Big(\frac{2t\theta}{\rho}- \frac{st}{2}\Big)
\cdot\frac{t^2}{2}+\Big(\frac{2t\theta}{\rho}-\frac{t^2}{2}\Big)\cdot\frac{t^2}{2}\right\}\left\|
\bm{\lambda} - \bm{\lambda}^*\right\|^2\notag\\
&\qquad\qquad + 
\Big(\frac{2t\theta}{\rho}-\frac{t^2}{2}\Big)
\cdot\frac{\kappa^2}{2t^2}\left\|
\dot{\bm{x}} 
\right\|^2.
\end{align}
Under the assumption of $t$ being sufficiently large so that the coefficient of the second order term is non-positive, we can write
\begin{align}
    &\Bigl(\frac{2t\theta}{\rho} - \frac{st}{2}\Bigr) \left\langle 
\bm{A}(\bm{x} - \bm{x}^*)
,
\bm{\lambda} - \bm{\lambda}^*
\right\rangle
+
\Big(\frac{2t\theta}{\rho}-\frac{t^2}{2}\Big)
\left\langle
  \bm{A}
\dot{\bm{x}} 
,
\bm{\lambda} - \bm{\lambda}^*
\right\rangle
\leq\Big(\frac{2t\theta}{\rho}- \frac{st}{2}\Big)\cdot\frac{\kappa^2}{2t^2}
\left\|
 \bm{x} - \bm{x}^{*}\right\|^2.
 \end{align}
Thus, \eqref{eq:21} yields
\begin{align}\label{eq:E'}
\mathcal{E'}(t)&\leq\Bigl(t - \frac{st\gamma}{2}\Bigr)\{\mathcal{L}(\bm{x}, \bm{\lambda}^*) - \mathcal{L}(\bm{x}^*, \bm{\lambda})\}\notag\\
&
+  t\Big(s + 1 - \frac{\alpha}{2}\Big) 
\left\| \dot{\bm{x}} \right\|^2
+ \Big(s(s + 1) + \xi - \frac{s\alpha}{2}\Big) 
\left\langle 
\dot{\bm{x}}, 
 \bm{x} - \bm{x}^{*} 
\right\rangle
+ \Big(\frac{2t\theta}{\rho}- \frac{st}{2}\Big)\cdot\frac{\kappa^2}{2t^2}
\left\|
 \bm{x} - \bm{x}^{*}\right\|^2
.\end{align}
Because we have
\begin{align*}
 {{t\left\|\dot{\bm{x}}\right\|^2}} = \frac{1}{t}\left\|s (\bm{x} - \bm{x}^{*}) + t\dot{\bm{x}}\right\|^2 - 2s\langle\dot{\bm{x}},  \bm{x} - \bm{x}^{*}\rangle - \frac{s^2}{t}\left\| \bm{x} - \bm{x}^{*}\right\|^2,
\end{align*}
it follows from \eqref{eq:E'} that
\begin{align}
    \mathcal{E'}(t)&\leq \Bigl(t - \frac{st\gamma}{2}\Bigr)\{ \mathcal{L}(\bm{x},\bm{\lambda}^*) - \mathcal{L}(\bm{x}^*,\bm{\lambda})\}\notag\\    
     &\quad+ \Big(s -\frac{\alpha}{2} + 1\Big) \left\{\frac{1}{t}\left\|s (\bm{x} - \bm{x}^{*}) + t\dot{\bm{x}}\right\|^2 - 2s{\langle\dot{\bm{x}},  \bm{x} - \bm{x}^{*}\rangle} - \frac{s^2}{t}\left\| \bm{x} - \bm{x}^{*}\right\|^2 \right\}\notag\\
     &\quad + \Big(s(s + 1) + \xi -\frac{s \alpha}{2}\Big) {\langle \dot{\bm{x}},  \bm{x} - \bm{x}^{*} \rangle}
      + \Big(\frac{2t\theta}{\rho}- \frac{st}{2}\Big)\cdot\frac{\kappa^2}{2t^2}\left\| \bm{x} - \bm{x}^{*}\right\|^2
  \notag\\
   &=\Bigl(t - \frac{st\gamma}{2}\Bigr)\{ \mathcal{L}(\bm{x},\bm{\lambda}^*) - \mathcal{L}(\bm{x}^*,\bm{\lambda})\}\notag\\
    &\quad+ \Big(s -\frac{\alpha}{2} + 1\Big) \frac{1}{t}\left\|s (\bm{x} - \bm{x}^{*}) + t\dot{\bm{x}}\right\|^2
+\left\{\Big(\frac{2t\theta}{\rho}- \frac{st}{2}\Big)\cdot\frac{\kappa^2}{2t^2} -\frac{s^2}{t}\Big(s-\frac{\alpha}{2} + 1\Big)\right\} \left\| \bm{x} - \bm{x}^{*}\right\|^2 
    \notag\\
    &\quad +\Big(-s^2-s+\xi  + \frac{s\alpha}{2}\Big){\langle\dot{\bm{x}},  \bm{x} - \bm{x}^{*}\rangle}\notag\\
 &=t\Bigl(1 - \frac{s\gamma}{2}\Bigr)\{ \mathcal{L}(\bm{x},\bm{\lambda}^*) - \mathcal{L}(\bm{x}^*,\bm{\lambda})\}
 + \frac{1}{t}\Bigl(s -\frac{\alpha}{2} + 1\Bigr)\left\|s (\bm{x} - \bm{x}^{*}) + t\dot{\bm{x}}\right\|^2\notag\\
&\quad+\frac{1}{t}\left\{\Big(\frac{2\theta}{\rho}- \frac{s}{2}\Big)\cdot\frac{\kappa^2}{2} -s^2\Big(s-\frac{\alpha}{2} + 1\Big)\right\}\left\| \bm{x} - \bm{x}^{*}\right\|^2 
  +\Big({\xi -s\Big(s -\frac{\alpha}{2} + 1\Big)}\Big){\langle\dot{\bm{x}},  \bm{x} - \bm{x}^{*}\rangle}.
\end{align}
Therefore, applying \ $\xi = s(s - \frac{\alpha}{2} + 1)$, we obtain
\begin{align}
  \mathcal{E}'(t)    
&\leq t\Bigl(1 - \frac{s\gamma}{2}\Bigr)\{ \mathcal{L}(\bm{x},\bm{\lambda}^*) - \mathcal{L}(\bm{x}^*,\bm{\lambda})\}
 + \frac{1}{t}\Bigl(s -\frac{\alpha}{2} + 1\Bigr) \left\|s (\bm{x} - \bm{x}^{*}) + t\dot{\bm{x}}\right\|^2\notag\\
&\quad+\frac{1}{t}\left\{\Big(\frac{2\theta}{\rho}- \frac{s}{2}\Big)\cdot\frac{\kappa^2}{2} -s^2\Big(s-\frac{\alpha}{2} + 1\Big)\right\}\left\| \bm{x} - \bm{x}^{*}\right\|^2.
\end{align}
\end{proof}

%% file: chapter5_lem2.tex
Here, to obtain the convergence rate explicitly, we set $\mathcal{H}(t) = t^p \mathcal{E}(t)$. Taking the derivative with respect to $t$, from \eqref{eq:Lyapunov} we get
\begin{align}\label{eq:H''}
  \mathcal{H}'(t) &= t^p\mathcal{E}'(t) + pt^{p-1}\mathcal{E}(t)\notag\\
&= t^{p-1}\Big[t\mathcal{E}'(t) + p\mathcal{E}(t)\Big]\notag\\
&\leq t^{p-1}\Big[t\Bigl(t\Big(1 - \frac{s\gamma}{2}\Big)\left\{\mathcal{L}(\bm{x},\bm{\lambda}^*)- \mathcal{L}(\bm{x}^*,\bm{\lambda})\right\}\notag\\
&\quad + \frac{1}{t}{\Big(s -\frac{\alpha}{2} + 1\Big)\left\|s(\bm{x} - \bm{x}^{*})+ t
\dot{\bm{x}}\right\|^2}
+ \frac{1}{t}\left\{\Big(\frac{2\theta}{\rho}- \frac{s}{2}\Big)\cdot\frac{\kappa^2}{2} -s^2\Big(s-\frac{\alpha}{2} + 1\Big)\right\}\left\|  \bm{x} - \bm{x}^{*}\right\|^2\Bigr)\notag\\
&\quad + p\left\{{{\frac{t^2}{2}\left\{\mathcal{L}(\bm{x},\bm{\lambda}^*)- \mathcal{L}(\bm{x}^*,\bm{\lambda})\right\}}}
 + {{\frac{1}{2} \left\| s(\bm{x} - \bm{x}^{*})+ t\dot{\bm{x}}\right\|^2 }}
 {{+ \frac{\xi}{2} \left\| \bm{x} - \bm{x}^{*} \right\|^2}}
+{{\theta \| \bm{\lambda} - \bm{\lambda}^*\|^2}}\right\}\Big]\notag\\
&= t^{p-1}\Big[{{t^2\Big(1-\frac{s\gamma}{2} + \frac{p}{2}\Big)\left\{\mathcal{L}(\bm{x},\bm{\lambda}^*)- \mathcal{L}(\bm{x}^*,\bm{\lambda})\right\}}}
+ {{\left\{\Big(s - \frac{\alpha}{2} + 1\Big) + \frac{p}{2}\right\}\left\| s(\bm{x} - \bm{x}^{*})+ t\dot{\bm{x}}\right\|^2}}
  \notag\\ &\quad{{+\Big(\left\{\Big(\frac{2\theta}{\rho}- \frac{s}{2}\Big)\cdot\frac{\kappa^2}{2} -s^2\Big(s-\frac{\alpha}{2} + 1\Big)\right\}+ \frac{p\xi}{2}\Big)\left\|\bm{x} - \bm{x}^{*}\right\|^2}} 
  + {{p\theta\| \bm{\lambda} - \bm{\lambda}^* \|^2}}\Big]\notag\\
  &= t^{p}\Big[{{\Big(1-\frac{s\gamma}{2} + \frac{p}{2}\Big)a(t)}}
  + \left\{\Big(s - \frac{\alpha}{2} + 1\Big) + \frac{p}{2}\right\}b(t)\notag\\
  &\quad+
  {{\underbrace{\Big(\left\{\Big(\frac{2\theta}{\rho}- \frac{s}{2}\Big)\cdot\frac{\kappa^2}{2} -s^2\Big(s-\frac{\alpha}{2} + 1\Big)\right\}+ \frac{p\xi}{2}\Big)}_{K_1}c(t)}}
  + {{\underbrace{p\theta}_{K_2}}} d(t)\Big].
\end{align}
 where
\begin{align*}
a(t) &= t\left\{\mathcal{L}(\bm{x},\bm{\lambda}^*)- \mathcal{L}(\bm{x}^*,\bm{\lambda})\right\},\\
b(t) &= \frac{1}{t}\left\| s(\bm{x} - \bm{x}^*)+ t\dot{\bm{x}}\right\|^2,\\
c(t) &= \frac{1}{t}\left\|\bm{x} - \bm{x}^{*} \right\|^2,\\
d(t) &=  \frac{1}{t}\| \bm{\lambda} - \bm{\lambda}^* \|^2.
\end{align*}
We choose here $s = \frac{\alpha}{\gamma+2},\ p = -2 + \frac{\gamma\alpha}{\gamma + 2}$.
The terms associated with $a(t)$ and $b(t)$ in \eqref{eq:H''} vanish and it appears that
\begin{align}\label{eq:H'}
  \mathcal{H}'(t)\leq t^{p}
  (K_1c(t) + K_2d(t))
  \leq t^{p}\max\{K_1, K_2\}(c(t) + d(t)),
\end{align}
where the real constant $\xi,K_1,K_2$ is given by:
\begin{align*}
  \xi &= s(s - \frac{\alpha}{2} + 1)\notag
  =\frac{-\alpha\gamma\left(\alpha-\frac{2(\gamma + 2)}{\gamma}\right)}{2(\gamma + 2)^2}, \\
K_1&=\left\{\Big(\frac{2\theta}{\rho}- \frac{s}{2}\Big)\cdot\frac{\kappa^2}{2} -s^2\Big(s-\frac{\alpha}{2} + 1\Big)\right\}+ \frac{p\xi}{2}\notag\\
&= \frac{\theta\kappa^2}{\rho} - \frac{\alpha(\alpha^2(\gamma-2)\gamma - 4\alpha(\gamma-1)(\gamma+2)+ (\gamma + 2)^2(\kappa^2 + 4))}{4(\gamma + 2)^3}\\
&=\frac{1}{4\rho(\gamma + 2)^3}\Bigl[4\theta\kappa^2\gamma^3 + \left\{24\theta\kappa^2-\alpha\rho(\alpha^2-4\alpha + (\kappa^2+4))\right\}\gamma^2
\\&\qquad\qquad+ \left\{48\theta\kappa^2+
2\alpha\rho(\alpha^2 + 2\alpha - 2(\kappa^2+4))\right\}\gamma + 32\theta\kappa^2 - 4\alpha\rho(2\alpha + (\kappa^2 + 4))\Bigr],\\
K_2 &= p\theta
 = \frac{\theta\gamma\left(\alpha-\frac{2(\gamma + 2)}{\gamma}\right)}{\gamma + 2}.
\end{align*}
From \eqref{eq:H'}, we can see the convergence behavior of the Lyapunov function by analyzing the signs of $K_1$ and $K_2$.

\begin{lemma}[Convergence behavior for parameters $\gamma$ and $\alpha$]\label{lem:sign}
Setting $\theta = \frac{\alpha\rho( (\kappa^2 + 4) + 2\alpha)}{8\kappa^2}$. 
The convergence behavior of $\mathcal{H}(t)$ is classified into the following cases depending on the signs of $K_1$ and $K_2$:
\begin{enumerate}[label=(\roman*), leftmargin=3.5em, labelsep=0.5em]
    \item[\textit{Case 1}] : When $\gamma > 2$ and $\alpha\in \left(0, \tilde{\alpha} \right]$, it holds that $\max\{K_1, K_2\}\leq 0$ (where $\tilde{\alpha} = \frac{2(\gamma+2)}{\gamma}$),
    \item[\textit{Case 2}] : Otherwise, $\max\{K_1, K_2\}> 0$,
  \end{enumerate}
Especially, \textit{Case 2} is divided into three subcases, 2.1, 2.2 and 2.3.
\begin{itemize}
  \item \textit{Case 2.1 : When $\gamma >2$, $\alpha\in\left(\tilde{\alpha}, \alpha_1\right]$, and $\alpha_2\le\alpha$, it holds that  $K_1\leq 0$ and $K_2> 0$,}
  \item \textit{Case 2.2 : When $\gamma = 2\ and\ \alpha\in\left(0, 4\right]$ or $\gamma\in\left[1,2\right)$ and $\alpha\in \left(0, \tilde{\alpha}\right]$, it holds that $K_1>0$ and $K_2\le 0$,}
  \item \textit{Case 2.3 : When $\gamma > 2$ and $\alpha\in\left(\alpha_1, \alpha_2\right)$ or $\gamma = 2$ and $\alpha>4$ or $\gamma\in\left[1,2\right)$ and $\alpha>\tilde{\alpha}$, it holds that $\min\{K_1, K_2\}>0$,}
\end{itemize}
where $\alpha_{1,2} = \tilde{\alpha}\cdot\frac{\gamma(\gamma+8)}{4(\gamma-2)}\pm\frac{\sqrt{D_1}}{2(\gamma-2)}\ (\alpha_1\le\alpha_2)$ and $D_1 =(\gamma + 2)^2\left\{\gamma^2 + 2(\kappa^2 + 12)\gamma - 4(\kappa^2 - 12)\right\}$.
\end{lemma}

\begin{proof}
Applying $\theta = \frac{\alpha\rho( (\kappa^2 + 4) + 2\alpha)}{8\kappa^2}$ into $K_1$ and $K_2$, we obtain
\begin{align*}
    \bar{K}_{1, \theta}(\alpha) &= \frac{\alpha\gamma}{8(\gamma + 2)^3}
    \Bigl[-2( \gamma - 2)\alpha^2 + 2(\gamma +2)(\gamma + 8)\alpha + (\gamma+2)^2(\kappa^2 +4)
    \Bigr],\\
    \bar{K}_{2, \theta}(\alpha) &= \frac{\alpha\gamma}{8(\gamma + 2)^3}
    \Bigl[\frac{\rho(\gamma+2)^2}{\kappa^2}(2\alpha + (\kappa^2 + 4))\left(\alpha - \frac{2(\gamma+2)}{\gamma}\right)\Bigr].
\end{align*}
Here, we put 
\begin{align*}
    K_{1, \theta}(\alpha) &= -2( \gamma - 2)\alpha^2 + 2(\gamma +2)(\gamma + 8)\alpha + (\gamma+2)^2(\kappa^2 +4),\\
    K_{2, \theta}(\alpha) &= \frac{\rho(\gamma+2)^2}{\kappa^2}(2\alpha + (\kappa^2 + 4))\left(\alpha - \frac{2(\gamma+2)}{\gamma}\right).
\end{align*}
Let $\tilde{\alpha} = \frac{2(\gamma+2)}{\gamma}$. It is clear that $\xi \le 0$ for $\alpha\ge \tilde{\alpha}$ and $\xi > 0$ for $\alpha\in (0,  \tilde{\alpha})$. 
$K_{2, \theta}(\alpha)\leq 0$ for $\alpha\in \left(0,  \tilde{\alpha} \right]$ and $K_{2, \theta}(\alpha) \ge 0$ for $\alpha \ge \tilde{\alpha}$.
When $\gamma = 2$, $K_{1, \theta}(\alpha)$ is always positive for all $\alpha$.\\
When $\gamma \neq 2$, as mentioned in Lemma 2, define $D_1$ as the discriminant of $K_{1, \theta}$,
\begin{align}\label{eq:D1}
    D_1 &= (\gamma+2)^2(\gamma+8)^2 + 2(\gamma-2)(\gamma + 2)^2(\kappa^2+4) \notag\\
    &=(\gamma + 2)^2\left\{\gamma^2 + 2(\kappa^2 + 12)\gamma - 4(\kappa^2 - 12)\right\}.
\end{align}
Solving  $K_{1, \theta}(\alpha) = 0$, we have
\begin{align*}
 \alpha_{1,2}= \frac{-(\gamma+2)(\gamma+8)\pm\sqrt{D_1}}{-2(\gamma-2)},\quad(\gamma\neq 2),
 \end{align*}
where $\alpha_1, \alpha_2$ are as follows:
 \begin{align*}
    \alpha_1 = \tilde{\alpha}\cdot\Big(\frac{\gamma(\gamma+8)}{4(\gamma-2)}+\frac{\gamma\sqrt{D_1}}{4(\gamma+2)(\gamma-2)}\Big),\quad \alpha_2 = \tilde{\alpha}\cdot\Big(\frac{\gamma(\gamma+8)}{4(\gamma-2)}-\frac{\gamma\sqrt{D_1}}{4(\gamma+2)(\gamma-2)}\Big).
\end{align*}
Here, the positive solution for the point where $D_1$ changes its sign is given by 
\begin{align*}
  d_1(\kappa^2) = -(\kappa^2 + 12) + \sqrt{(\kappa^2 + 12)^2 + 4(\kappa^2 - 12)}\ (>0).
\end{align*}
The sign of the discriminant $D_1$ depends on the parameters as follows:
\begin{itemize}
    \item When $\gamma > 2$, from \eqref{eq:D1}, we have $D_1 > 0$.
    \item When $\gamma \in [1, 2)$, the sign of $D_1$ depends on $\kappa^2$. We can see the following from \eqref{eq:D1}.
    \begin{itemize}
        \item For $d_1(\kappa^2) \le 1$ (i.e., $\kappa^2 < 36.5$), $D_1 \ge 0$.
        \item For $1 < d_1(\kappa^2) < 2$, $D_1 < 0$ when $\gamma \in [1, d_1(\kappa^2))$, whereas $D_1 \ge 0$ when $\gamma \in [d_1(\kappa^2), 2)$.
        \item For $d_1(\kappa^2) >2$, $D_1 < 0$.
     \end{itemize}
      \end{itemize}
The relationship between $\alpha_1, \alpha_2$ and $\tilde{\alpha}$, it is straightforward to see that $\tilde{\alpha} <\alpha_1<\alpha_2$ for $\gamma>2$.
When $\gamma\in[1,2)$, to compare the values of $\alpha_2$ and $\tilde{\alpha}$ in the case of $D_1 \ge 0$, we put 
$\Gamma = \Big(\frac{\gamma(\gamma+8)}{4(\gamma-2)}-\frac{\gamma\sqrt{D_1}}{4(\gamma+2)(\gamma-2)}\Big) = \frac{\gamma\{(\gamma + 8) - \sqrt{\gamma^2 + 2(\kappa^2 + 12)\gamma - 4(\kappa^2 - 12)}\}}{4(\gamma - 2)}$
and find $\Gamma$ is negative in both cases, thus $\alpha_1 \le \alpha_2 (< 0) < \tilde{\alpha}$.
Hence, when $\gamma>2$, $K_{1, \theta}(\alpha)\le0$ for $0 < \alpha \le \alpha_1, \alpha_2 \le \alpha$ and $K_{1, \theta }(\alpha)\ge0$ for $\alpha_1 \le\alpha \le\alpha_2$. 
When $\gamma\in\left[1,2\right)$, $K_{1, \theta}(\alpha)>0$ for any $\alpha$.
Given the above, 
\begin{itemize}
  \item \textit{Case 1 :When $\gamma > 2$ and $\alpha\in \left(0, \tilde{\alpha} \right]$, it holds that $\max\{K_1, K_2\}\leq 0$,}
  \item \textit{Case 2.1 : When $\gamma >2$ and $\alpha\in\left(\tilde{\alpha}, \alpha_1\right], \alpha_2\le\alpha$, it holds that  $K_1\leq 0$ and $K_2> 0$,}
  \item \textit{Case 2.2 : When $\gamma = 2\ and\ \alpha\in\left(0, 4\right]$ or $\gamma\in\left[1,2\right)$ and $\alpha\in \left(0, \tilde{\alpha}\right]$, it holds that $K_1>0$ and $K_2\le 0$,}
  \item \textit{Case 2.3 : When $\gamma > 2$ and $\alpha\in\left(\alpha_1, \alpha_2\right)$ or $\gamma = 2$ and $\alpha>4$ or $\gamma\in\left[1,2\right)$ and $\alpha>\tilde{\alpha}$, it holds that $\min\{K_1, K_2\}>0$.}
\end{itemize}
\end{proof}
Table \ref{tab:sign-relation} summarizes the sign relations of $K_1$ and $K_2$ for each range of $\gamma$.
\begin{table}[ht]
\centering
\scriptsize
\setlength{\tabcolsep}{4pt} 
\caption{Sign relations of $K_1$ and $K_2$ for various ranges of $\gamma, \kappa^2, \alpha$}
\label{tab:sign-relation}
\renewcommand{\arraystretch}{1.5} 
\begin{tabular}{lllllll} 
\toprule
$\gamma$ & 1. $\max\{K_1, K_2\}\leq 0$ & 2.1. $K_1\leq 0, K_2> 0$ & 2.2. $K_1> 0, K_2\le 0$ & 2.3. $\min\{K_1, K_2\}>0$ \\ 
\midrule

$\gamma>2$ & $\alpha\in (0, \tilde{\alpha}]$ & $\alpha\in\left(\tilde{\alpha}, \alpha_1\right], \alpha_2\le\alpha$ &$-$ & $\alpha\in(\alpha_1, \alpha_2)$ \\
\addlinespace[1ex]

\midrule
$\gamma=2$  & $-$ & $-$ & $\alpha\in(0, 4]$ & $\alpha>4$ \\
\addlinespace[1ex]

\midrule
$\gamma\in[1,2)$ & $-$ & $-$& $\alpha\in (0, \tilde{\alpha}]$ & $\alpha>\tilde{\alpha}$ \\ 
\bottomrule
\end{tabular}
\end{table}

%% file: chapter5_thm.tex
Here, we define the rate function $\textit{rate}(\alpha, \gamma)$ as an exponent depending on parameters $\alpha$ and $\gamma$. The convergence rate is then given by $\mathcal{O}(1/t^{\textit{rate}(\alpha, \gamma)})$.

\begin{theorem}[Convergence rate under \textbf{H1}$(\gamma)$]\label{thm:rate}
Assume that \textbf{H1}$(\gamma)$ holds with $\gamma \geq 1$. 
For sufficiently large $t$, let the parameters $\alpha$ and $\rho$ be chosen such that the condition $\theta = \frac{\alpha\rho( (\kappa^2 + 4) + 2\alpha)}{8\kappa^2}\ (\kappa^2>0)$ is satisfied. 
Then, the convergence rate is given by
\begin{align*}
  F(\bm{x}(t)) - F^* = \mathcal{O}\Big(\frac{1}{t^{\frac{\gamma\alpha}{\gamma + 2}}}\Big),
\end{align*}
where $\textit{rate}(\alpha, \gamma)$ is given by
\begin{align*}
    \textit{rate}(\alpha, \gamma) := \frac{\gamma\alpha }{\gamma + 2} \quad\text{when } \gamma > 2 \text{ and }\alpha \in (0, \tilde{\alpha}].
 \end{align*}
 Especially, when $\gamma > 2$ and $\alpha = \tilde{\alpha}$, the convergence rate is given by
\begin{align*}
  F(\bm{x}(t)) - F^* = \mathcal{O}\Big(\frac{1}{t^2}\Big).
\end{align*}
\end{theorem}

\begin{proof}
When $\max\{K_1, K_2\}\leq 0$ (\textit{Case 1}), $\mathcal{H}$ is monotonically decreasing and bounded, namely,
\begin{align}\label{eq:b}
  \forall t \geq t_0,\ \mathcal{H}(t) \leq \mathcal{H}(t_0).
\end{align}
From the range of $\alpha$, it is clear that $\xi \geq 0$. Using \eqref{eq:Lyapunov}, we have
\begin{align}\label{eq:bb}
  \mathcal{E}(t)\geq \frac{t^2}{2}\left\{\mathcal{L}(\bm{x},\bm{\lambda}^*)- \mathcal{L}(\bm{x}^*,\bm{\lambda})\right\}\geq 0.
\end{align}
Using $\mathcal{H}(t)=t^p\mathcal{E}(t)$, \eqref{eq:bb} can be written as follows:
\begin{align}\label{eq:bbb}
  \mathcal{H}(t)\geq \frac{t^{p+2}}{2}\left\{\mathcal{L}(\bm{x},\bm{\lambda}^*)- \mathcal{L}(\bm{x}^*,\bm{\lambda})\right\}.
\end{align}
From \eqref{eq:b},\eqref{eq:bbb}, we obtain 
\begin{align}\label{eq:h0}
  \frac{t^{p+2}}{2}\left\{\mathcal{L}(\bm{x},\bm{\lambda}^*)- \mathcal{L}(\bm{x}^*,\bm{\lambda})\right\}\leq \mathcal{H}(t)\leq\mathcal{H}(t_0).
\end{align} 
From the definition of the Lagrange function, the following holds.
\begin{align*}
       \frac{t^{p+2}}{2}\Big( F(\bm{x}(t)) - F^*\Big)+\frac{t^{p+2}}{2}\langle \bm{\lambda}^*, A \bm{x} - \bm{b}\rangle  & \leq\mathcal{H}(t_0),
\end{align*}
which implies 
       \begin{align*}
  F(\bm{x}(t)) - F^* = \mathcal{O}\Big(\frac{1}{t^{p+2}}\Big) = \mathcal{O}\Big(\frac{1}{t^{\frac{\gamma\alpha}{\gamma + 2}}}\Big).
\end{align*}
When $\textit{rate}(\alpha, \gamma) := \frac{\gamma\alpha }{\gamma + 2} \ge 2 \Leftrightarrow \alpha \ge \frac{2(\gamma + 2)}{\gamma} = \tilde{\alpha}$, we have the worst-case convergence rate of $\mathcal{O}(1/t^2)$.
In \textit{Case 1}, since we can take $\alpha = \tilde{\alpha}$ for any $\gamma > 2$, we can achieve the $\mathcal{O}(1/t^2)$ rate for any $\gamma > 2$ whenever $\alpha = \tilde{\alpha}$ is chosen, but in other cases, this rate cannot be achieved.
\end{proof}

\begin{theorem}[Convergence rate under \textbf{H1}$(\gamma)$ and \textbf{H2}$(r)$]\label{thm:rate}
Assume that \textbf{H1}$(\gamma)$ holds with $\gamma \geq 1$ and \textbf{H2}$(r)$ holds with $1\le r\le 2$.
For sufficiently large $t$, let the parameters $\alpha$ and $\rho$ be chosen such that the condition $\theta = \frac{\alpha\rho( (\kappa^2 + 4) + 2\alpha)}{8\kappa^2}\ (\kappa^2>0)$ is satisfied. 
Then, the convergence rate is given by
\begin{align*}
  F(\bm{x}(t)) - F^* = \mathcal{O}\Big(\frac{1}{t^{\frac{\gamma\alpha}{\gamma + 2}}}\Big),
\end{align*}
where $\textit{rate}(\alpha, \gamma)$ is given by
\begin{align*}
    \textit{rate}(\alpha, \gamma) := 
\begin{cases}
\displaystyle \frac{\gamma\alpha }{\gamma + 2} & \text{when } \gamma > 2 \text{ and }\alpha >\tilde{\alpha}, \\
\displaystyle \frac{\gamma\alpha }{\gamma + 2} & \text{when } \gamma = 2 \text{ and }\alpha >0, \\
\displaystyle \frac{\gamma\alpha }{\gamma + 2} & \text{when } \gamma \in [1,2)\text{ and } \alpha >0.
\end{cases}
\end{align*}
In particular, by choosing $\alpha \ge \tilde{\alpha}$ for any $\gamma \ge 1$, we obtain an improved convergence rate of at least
\begin{align*}
  F(\bm{x}(t)) - F^* = \mathcal{O}\Big(\frac{1}{t^{2}}\Big).
\end{align*}
\end{theorem}

\begin{proof}
When $\max\{K_1, K_2\}>0$ (\textit{Case 2}), the monotonicity of $\mathcal{H}$ does not hold. 
For \textit{Case 2.1} and \textit{Case 2.3}, we have $\xi<0$.
By using~\eqref{lemma:growth}, we have
\begin{align*}
  Kt\Bigl\|\begin{pmatrix}
    \bm{x}\\
    \bm{\lambda}
  \end{pmatrix}
   - \begin{pmatrix}
    \bm{x}^*\\
    \bm{\lambda}^*
  \end{pmatrix}\Bigr\|^2
  \leq 
  Kt\Bigl\|\begin{pmatrix}
    \bm{x}\\
    \bm{\lambda}
  \end{pmatrix}
   - \begin{pmatrix}
    \bm{x}^*\\
    \bm{\lambda}^r
  \end{pmatrix}\Bigr\|^r
  \leq t\left\{\mathcal{L}(\bm{x},\bm{\lambda}^*)- \mathcal{L}(\bm{x}^*,\bm{\lambda})\right\},
  \end{align*}
  which yields
  \begin{align*}
Kt^2\cdot\frac{1}{t}\bigl(\|\bm{x} - \bm{x}^*\|^2 + \|\bm{\lambda} - \bm{\lambda}^*\|^2\bigr)
  &\leq t\left\{\mathcal{L}(\bm{x},\bm{\lambda}^*)- \mathcal{L}(\bm{x}^*,\bm{\lambda})\right\}.
\end{align*}
Thus, we obtain
\begin{align}\label{eq:c+d}
  c(t) \leq c(t) + d(t) \leq\frac{1}{Kt^2}a(t)
.\end{align}
It follows from \eqref{eq:Lyapunov} that
\begin{align*}
  \mathcal{H}(t) = t^p\mathcal{E}(t)&\geq\frac{t^{p+1}}{2}\left[a(t)
    + \xi c(t)\right]\notag\geq\frac{t^{p+1}}{2}\left[\Bigl(1
    + \frac{\xi}{Kt^2}\Bigr)a(t)\right].
  \end{align*}
For sufficiently large $t$, it holds that 
\begin{align*}
1 + \frac{\xi}{Kt^2} > \frac{1}{2},
\end{align*}
which implies that there exists $t_1$ such that for all $t \geq t_1, \mathcal{H}(t)\geq0$ and
\begin{align}\label{eq:4}
  \mathcal{H}(t)&\geq \frac{t^{p+1}}{4}a(t).
\end{align}
From \eqref{eq:H'},\eqref{eq:c+d},\eqref{eq:4}, we get
\begin{align*}
  \mathcal{H}'(t)&\leq t^{p}
  (K_1c(t) + K_2d(t))\notag\\
  &\leq t^{p}\max\{K_1, K_2\}(c(t) + d(t))\notag\\
  &\leq t^{p-2}\frac{\max\{K_1, K_2\}}{K}a(t)\notag\\
 &\leq \left[\frac{\max\{K_1, K_2\}}{K}\right]\frac{4\mathcal{H}(t)}{t^{3}}.
\end{align*}
From the differential form of \textit{Proximal Gr\"onwall Inequality} (\cite{Clarke1998} Theorem 11.29.),  $\text{for all } t \ge t_1$, there exists a constant such that
\begin{align}\label{eq:gronwall}
  \mathcal{H}(t) \le \mathcal{H}(t_1)\exp\left(\int_{t_1}^{t}\left[\frac{2\max\{K_1, K_2\}}{Ks^3}\right]ds\right).
\end{align}
From \eqref{eq:4}, \eqref{eq:gronwall} and the definition of the Lagrange function, we conclude that $t^{p+2}\Big( F(\bm{x}(t)) - F^*\Big)+t^{p+2}\langle \bm{\lambda}^*, A \bm{x} - \bm{b}\rangle$ is bounded. Thus, we have
\begin{align*}
F(\bm{x}(t)) - F^* = \mathcal{O}\Big(\frac{1}{t^{p+2}}\Big) = \mathcal{O}\Big(\frac{1}{t^{\frac{\gamma\alpha}{\gamma + 2}}}\Big).
\end{align*}
For \textit{Case 2.2}, we have $\xi\ge0$. It follows from \eqref{eq:Lyapunov} that
\begin{align}\label{eq:H_}
  \mathcal{H}(t) = t^p\mathcal{E}(t)\geq\frac{t^{p+1}}{2}a(t).
  \end{align}
From \eqref{eq:H'},\eqref{eq:c+d}, \eqref{eq:H_}, we get
\begin{align*}
  \mathcal{H}'(t)&\leq t^{p}
  (K_1c(t) + K_2d(t))\notag\\
  &\leq t^{p}\max\{K_1, K_2\}(c(t) + d(t))\notag\\
  &\leq t^{p-2}\frac{\max\{K_1, K_2\}}{K}a(t)\notag\\
 &\leq \left[\frac{\max\{K_1, K_2\}}{K}\right]\frac{2\mathcal{H}(t)}{t^{3}}.
\end{align*}
From the differential form of \textit{Proximal Gr\"onwall Inequality} (\cite{Clarke1998} Theorem 11.29.),  $\text{for all } t \ge t_1$, there exists a constant such that
\begin{align}\label{eq:gronwall_}
  \mathcal{H}(t) \le \mathcal{H}(t_1)\exp\left(\int_{t_1}^{t}\left[\frac{2\max\{K_1, K_2\}}{Ks^3}\right]ds\right).
\end{align}
From \eqref{eq:H_}, \eqref{eq:gronwall_} and the definition of the Lagrange function, we conclude that $t^{p+2}\Big( F(\bm{x}(t)) - F^*\Big)+t^{p+2}\langle \bm{\lambda}^*, A \bm{x} - \bm{b}\rangle$ is bounded. Thus, we have
\begin{align*}
F(\bm{x}(t)) - F^* = \mathcal{O}\Big(\frac{1}{t^{p+2}}\Big) = \mathcal{O}\Big(\frac{1}{t^{\frac{\gamma\alpha}{\gamma + 2}}}\Big).
\end{align*}
When $\textit{rate}(\alpha, \gamma) := \frac{\gamma\alpha }{\gamma + 2} \ge 2 \Leftrightarrow \alpha \ge \frac{2(\gamma + 2)}{\gamma} = \tilde{\alpha}$, we have the worst-case convergence rate of $\mathcal{O}(1/t^2)$.
In \textit{Case 2}, since $\alpha$ can be selected to satisfy $\alpha \ge \tilde{\alpha}$ for any $\gamma \ge 1$, choosing such an $\alpha$ yields a faster rate of at least $\mathcal{O}(1/t^2)$.
\end{proof}
The relationship between the convergence rate and $\alpha$ can also be viewed in terms of the number of iterations.
As can be inferred from \eqref{eq:h0}, \eqref{eq:gronwall}, the convergence coefficient is given by :
\begin{align}\label{eq:number}
  \mathcal{H}_0 &= t_0^p\mathcal{E}(t_0)\notag\\
  &= \frac{t_0^p}{2} \left[ t_0^{2}\left( \mathcal{L}(\bm{x}(t_0), \bm{\lambda}^*) - \mathcal{L}(\bm{x}^*, \bm{\lambda}(t_0)) \right) \right. \notag \\
  &\quad + \left\| s(\bm{x}(t_0) - \bm{x}^*) + t_0\dot{\bm{x}}(t_0)\right\|^2 
  + \xi \left\| \bm{x}(t_0) - \bm{x}^* \right\|^2  \left. + 2\theta \| \bm{\lambda}(t_0) - \bm{\lambda}^* \|^2 \right].
\end{align}
Here, the coefficients $s = \frac{\alpha}{\gamma+2},\ \xi = \frac{-\alpha\gamma\left(\alpha-\tilde{\alpha}\right)}{2(\gamma + 2)^2}$, and $\theta = \frac{\alpha\rho( (\kappa^2 + 4) + 2\alpha)}{8\kappa^2}$
directly impacts the convergence rate. Since the condition $\xi \ge 0$ must be satisfied in \textit{Case 1} and \textit{Case 2.2}, the optimal $\alpha$ that minimizes $\xi$ is given by $\tilde{\alpha}$. Conversely, in \textit{Case 2.1} and \textit{Case 2.3}, where $\xi < 0$, selecting a larger $\alpha (\ge \tilde{\alpha})$ minimizes $\xi$.
Regarding $\rho$, two main observations can be made.
First, in the context of Lyapunov stability analysis, increasing $\rho$ enlarges the positive coefficient $K_2$ in $\mathcal{H}(t)'$ when $\alpha \ge \tilde{\alpha}$, which reduces the decay rate of the Lyapunov function and thereby slows down convergence. 
Conversely, when $\alpha < \tilde{\alpha}$, increasing $\rho$ decreases the negative coefficient $K_2$ in $\mathcal{H}(t)'$, leading to faster convergence. 
Second, in terms of the number of iterations, \eqref{eq:number} also reveals a trade-off between $\alpha$ and $\rho$.
\begin{remark}[Insights on Parameters $\alpha$ and $\rho$]
As established in the previous section, while the range of $\alpha$ was estimated earlier based on the convergence rate, 
analyzing the iteration count yields a consistent estimate for $\alpha$. 
Furthermore, regarding $\rho$, the coefficients in the Lyapunov analysis provide additional insights for setting $\rho$.
\end{remark}

%% file: chapter7.tex
\section{Discretization of ODI corresponding to Accelerated Primal-Dual Algorithms}
Next, we consider the discretization of \eqref{eq:dynamics}. To begin with, we rewrite the second-order ODI for the primal variables $\bm{x}$ in \eqref{eq:dynamics} as first-order ODI by introducing an auxiliary variable $\bm{v}$ such that $\dot{\bm{x}} = \bm{v} - \bm{x}$.
\begin{equation}\label{eq:v}
\begin{cases}
 \displaystyle \dot{\bm{x}} = \bm{v} - \bm{x},\\
 \displaystyle 0 \in 2\dot{\bm{v}} + \frac{\alpha}{t} (\bm{v} - \bm{x}) + \partial_{\bm{x}} \mathcal{L}(\bm{x}, \bar{\bm{\lambda}}),\\
 \displaystyle 0 = \frac{\rho}{t}\dot{\bm{\lambda}} - \nabla_{\bm{\lambda}}\mathcal{L}(\bm{v}, \bm{\lambda}).
\end{cases}
\end{equation}

Indeed, there are various discretization schemes available, each of which can be discussed from the perspective of numerical analysis.
However, we restrict our attention to the implicit Euler method in this paper.

The implicit Euler discretization of \eqref{eq:v} with step size $h > 0$ and $x_k \approx x(kh)$ is given by
\begin{equation}\label{eq:39}
 \begin{cases}
 \displaystyle \frac{\bm{x}_{k+1} - \bm{x}_k}{h} = \bm{v}_{k+1} - \bm{x}_{k+1}, \\[7pt]
 \displaystyle 0 \in \frac{2(\bm{v}_{k+1}-\bm{v}_k)}{h} + \frac{\alpha}{kh} (\bm{v}_{k+1} - \bm{x}_{k+1}) + \partial_{\bm{x}}\mathcal{L}(\bm{x}_{k+1}, \bar{\bm{\lambda}}_{k+1}),\\[7pt]
 \displaystyle 0 = \frac{\rho}{kh} \cdot\frac{\bm{\lambda}_{k+1} - {\bm{\lambda}}_k}{h} - \nabla_{\bm{\lambda}}\mathcal{L}(\bm{v}_{k+1}, {\bm{\lambda}}_{k+1}).
 \end{cases}
\end{equation}
Eliminating $\bm{v}_{k+1}$ from the second equation in \eqref{eq:39} by using the first equation, we get
\begin{equation}\label{eq:vv}
  \begin{cases}
  \displaystyle \frac{\bm{x}_{k+1} - \bm{x}_k}{h} = \bm{v}_{k+1} - \bm{x}_{k+1}, \\
  \displaystyle 0 \in \partial_{\bm{x}}\mathcal{L}(\bm{x}_{k+1}, \bar{{\bm{\lambda}}}_{k+1}) + \frac{2k(1+h) + \alpha}{kh^2}(\bm{x}_{k+1} - \tilde{\bm{x}}_k),\ \tilde{\bm{x}}_k := \bm{x}_k + \frac{2kh}{2k(1+h) + \alpha} (\bm{v}_k - \bm{x}_{k}),\\
  \displaystyle 0 = \frac{\rho}{kh} \cdot \frac{{\bm{\lambda}}_{k+1} - {\bm{\lambda}}_k}{h} - \nabla_{\bm{\lambda}}\mathcal{L}(\bm{v}_{k+1}, {\bm{\lambda}}_{k+1}).
  \end{cases}
\end{equation}
Rearranging \eqref{eq:vv}, we obtain the following primal-dual formulation:
\begin{equation}\label{eq:vvv}
\begin{cases}
 \displaystyle \bm{x}_{k+1} = \arg\min_{\bm{x}}\Biggl\{\mathcal{L}(\bm{x}, \bar{{\bm{\lambda}}}_{k+1}) + \frac{2k(1+h) + \alpha}{kh^2}\left\|\bm{x} - \tilde{\bm{x}}_k\right\|^2\Biggr\},\ \tilde{\bm{x}}_k := \bm{x}_k + \frac{2kh}{2k(1+h) + \alpha} (\bm{v}_k - \bm{x}_{k}),\\
 \displaystyle \bm{v}_{k+1} = \bm{x}_{k+1} + (\bm{x}_{k+1} - \bm{x}_k)/h, \\
 \displaystyle {\bm{\lambda}}_{k+1} = {\bm{\lambda}}_k + \frac{h^2 k}{\rho}(\bm{A}\bm{v}_{k+1} - \bm{b}).
\end{cases}
\end{equation}
From here, for simplicity, we consider the case $m = 2$ and denote $\bm{x} = \begin{pmatrix}x^1\\x^2\end{pmatrix} = \begin{pmatrix}x\\y\end{pmatrix}$,\  
$\bm{v} = \begin{pmatrix}v^1\\v^2\end{pmatrix} = \begin{pmatrix}v\\w\end{pmatrix}$,\ 
$\begin{pmatrix}f^1\\f^2\end{pmatrix} = \begin{pmatrix}f\\g\end{pmatrix}$, and\   
$\begin{pmatrix}A^1\\A^2\end{pmatrix} = \begin{pmatrix}A\\B\end{pmatrix}$.
Note that the scheme in \eqref{eq:vvv} is formal at this stage, since the auxiliary variable $\bar{\lambda}_{k+1}$ has not yet been specified.
Motivated by the idea of \textit{semi-implicit splitting} proposed in \cite[Section~3.2]{Luo2023}, we adopt the same semi-implicit definition of the auxiliary dual variable $\bar{\lambda}_{k+1}$ as follows:
\begin{equation}\label{eq:lambda_bar_update}
\bar{\lambda}_{k+1}
:= \lambda_k + \frac{h^2k}{\rho}\bigl(A v_{k+1} + B w_k - b\bigr).
\end{equation}
This choice preserves the contraction property of the discrete Lyapunov function while decoupling the primal updates and yields the identity
\begin{equation}\label{eq:lambda_difference}
\lambda_{k+1}-\bar{\lambda}_{k+1}
= \frac{h^2 k}{\rho} B\bigl(w_{k+1}-w_k\bigr),
\end{equation}
which ensures that the dual difference is directly controlled by the primal increment $\|w_{k+1}-w_k\|$, and therefore does not destroy the contraction property of the discrete Lyapunov function.

Following \cite[Section~3.2]{Luo2023}, we define $\hat{\lambda}_{k}$ as follows:
\begin{equation}\label{eq:lambda_hat}
\hat{\lambda}_{k} = \lambda_k - \rho^{-1}(Ax_k + By_k - b) + \frac{h^2 k}{\rho} B(w_k - y_k).
\end{equation}

By reformulating \eqref{eq:vvv}, we obtain the following algorithm:
\begin{align}\label{eq:vvvv}
\begin{cases}
\widehat{\lambda}_{k} = \lambda_k - \rho^{-1}(Ax_k + By_k - b) + \frac{h^2 k}{\rho} B(w_k - y_k),\\
\displaystyle x_{k+1} = \arg\min_{x}\Biggl\{\mathcal{L}_{\sigma}(\boldsymbol{x}, y_k, \hat{\lambda}_{k}) + \frac{2k(1+h) + \alpha}{kh^2}\|\boldsymbol{x} - \tilde{x}_k\|^2\Biggr\},\\
\displaystyle v_{k+1} = x_{k+1} + (x_{k+1} - x_k)/h, \\
\displaystyle \bar{\lambda}_{k+1} = \lambda_k + \frac{h^2 k}{\rho}(Av_{k+1} + Bw_{k} - b),\\
\displaystyle y_{k+1} = prox_{\tau g}(\tilde{y}_k -\tau_k B^\top\bar{\lambda}_{k+1}),\\
\displaystyle w_{k+1} = y_{k+1}+(y_{k+1} - y_k)/h,\\
\displaystyle \lambda_{k+1} = \lambda_k + \frac{h^2 k}{\rho}(Av_{k+1} + Bw_{k+1} - b).\\
\end{cases}
\end{align}
where $\tilde{x}_k:= x_k + \frac{2kh}{2k(1+h) + \alpha} (v_k - x_{k}),\ 
\tilde{y}_k:= y_k + \frac{2kh}{2k(1+h) + \alpha} (w_k - y_{k}),$
$\sigma := \rho^{-1},\ \tau_k := \frac{kh^2}{2(2k(1+h) + \alpha)},$
and $\mathcal{L}_{\sigma}(\boldsymbol{x}, y_k, \hat{\lambda}_k) := f(\boldsymbol{x}) + g(y_k) + \hat{\lambda}_k^\top(A\boldsymbol{x} + By_k -b) + \frac{\sigma}{2}\|A\boldsymbol{x} + By_k - b\|^2$.

\subsection{Numerical Experiments}
In this section, we validate the theoretical results of the proposed method through numerical experiments on a constrained convex optimization problem arising from sparse recovery.
Specifically, we consider the following $\ell_1$-regularized least absolute deviation (LAD) problem, which is widely used in compressed sensing to recover original signals from limited observations~\cite{Lustig}. This approach is particularly effective for underdetermined systems, such as those encountered in MRI, astronomy, etc.
The LAD problem is given by the form:
\begin{equation}\label{eq:LAD_original}
\min_{x\in\mathbb{R}^n}
\; \delta \|x\|_1 + \|Ax - b\|_1,
\end{equation}
where $A \in \mathbb{R}^{m\times n}$ and $b \in \mathbb{R}^m$ denote the measurement matrix and the observation vector, respectively, and $\delta>0$ is a regularization parameter promoting sparsity of the solution. 
This setting corresponds to the case where $H_1(\gamma)$ holds with $\gamma = 1$ and $H_2(r)$ holds with $r=1$.

To fit the framework of the proposed algorithm, we reformulate~\eqref{eq:LAD_original} as an equivalent constrained problem by introducing an auxiliary variable $y$:
\begin{equation}\label{eq:LAD_constrained}
\begin{aligned}
\min_{x\in\mathbb{R}^n,\; y\in\mathbb{R}^m}
\quad & \delta \|x\|_1 + \|y\|_1, \\
\text{s.t.}\quad & Ax - b = y.
\end{aligned}
\end{equation}
This reformulation allows us to apply the proposed primal--dual algorithm~\eqref{eq:vvvv} directly.

\paragraph{Experimental settings.}
The problem dimensions are set to $m=200$ and $n=100$, and the ground-truth signal is assumed to be $s$-sparse with $s=20$.
The entries of the matrix $A$ are generated with density $0.3$, and the observation vector $b$ is contaminated with additive noise of level $0.01$.
Unless otherwise stated, the regularization parameter is fixed to $\delta = 0.1$.

The proposed method initializes both the primal and dual variables with random values drawn from a standard normal distribution.
We set the step size $h$ to $1/\kappa$ and vary the acceleration parameters
$\alpha$ and $\rho$ in each experiment to explore optimal values.
From a theoretical perspective, the penalty parameter $\sigma$ appearing in the $x$-update is chosen as $\sigma=\rho^{-1}$, in accordance with the derivation of the discretization scheme.
In practice, however, to enhance numerical accuracy and convergence speed, we update $\sigma_k$ instead of using a fixed $\sigma$ to vary adaptively during the iterations.
Specifically, while the initial value is set to $\sigma_0=\rho^{-1}$, the subsequent updates of $\sigma_k$ follow the adaptive penalty strategy described in \cite[Section~3.4.1]{Boyd2011}, commonly referred to as the \textit{varying penalty parameter} with $(\tau, \mu) = (2, 10)$.\\
\begin{enumerate}
  \item To verify our theoretical convergence analysis and the optimality of $\alpha$ and $\rho$, we conducted experiments by varying $\alpha \in \{4.0, 6.0, 10.0, 20.0, 30.0\}$ while keeping $\rho$ fixed to values in $\{2.0, 3.0, 4.0, 8.0\}$. 
This scenario corresponds to \textit{Case 2} with $\gamma \in [1, 2)$, where the optimal convergence rate of $\mathcal{O}(1/t^2)$ is theoretically expected for $\alpha \ge \tilde{\alpha} \, (= 6.0)$. 
The value $\alpha = 4.0$ is also included to check the behavior when $\alpha < \tilde{\alpha}$. 
Under these conditions, the convergence rate is anticipated to fall within the range of $[2.0, 3.0)$.

The experimental results shown in Figs.~(a)–(d) confirmed our theoretical predictions and provided the following key insights:
\begin{itemize}
    \item \textbf{Impact of Small $\rho$ and Stability}: As observed in Fig.~(a), a small value of $\rho$ leads to numerical instability (or computational failure). Although convergence was achieved in all other cases, increasing $\rho$ generally reduced the overall convergence speed, highlighting a trade-off between $\alpha$ and $\rho$.
    
    \item \textbf{Optimal Rate Realization}: Figs.~(b) and (c) indicate a tendency that taking a larger $\alpha$ ($\alpha \ge \tilde{\alpha}$) leads to faster convergence.
    
    \item \textbf{Lyapunov Coefficient Behavior under Large $\rho$}: Fig.~(d), which uses the largest value of $\rho$, clearly reflects our Lyapunov analysis. When $\alpha \ge \tilde{\alpha}$, increasing $\rho$ enlarges the positive coefficient $K_2$, which slows down convergence. Conversely, for $\alpha\ (= 4.0) < \tilde{\alpha}$, where $K_2 < 0$, increasing $\rho$ shifts $K_2$ further in the negative direction, thereby accelerating convergence.
\end{itemize}
  \item Next, we compare the proposed method using the parameter setting $(\alpha, \rho) = (0.1, 2.0)$ with the Linearized-Proximal-ADMM (AD-LPMM), following the notation as detailed in \cite{amir}.
While both our method and the AD-LPMM converged, the AD-LPMM exhibited a faster convergence rate. 
This is due to the fact that a stability analysis regarding the algorithm parameters $\rho$ and step size $h$ remains unexplored, 
preventing optimal tuning. Conducting a discrete-time analysis is a subject for our future work [See Figure \ref{fig:right-group}].
\end{enumerate}

\begin{figure}[htbp]
  \centering

  \begin{subfigure}{0.48\textwidth}
    \centering
    \includegraphics[width=\textwidth]{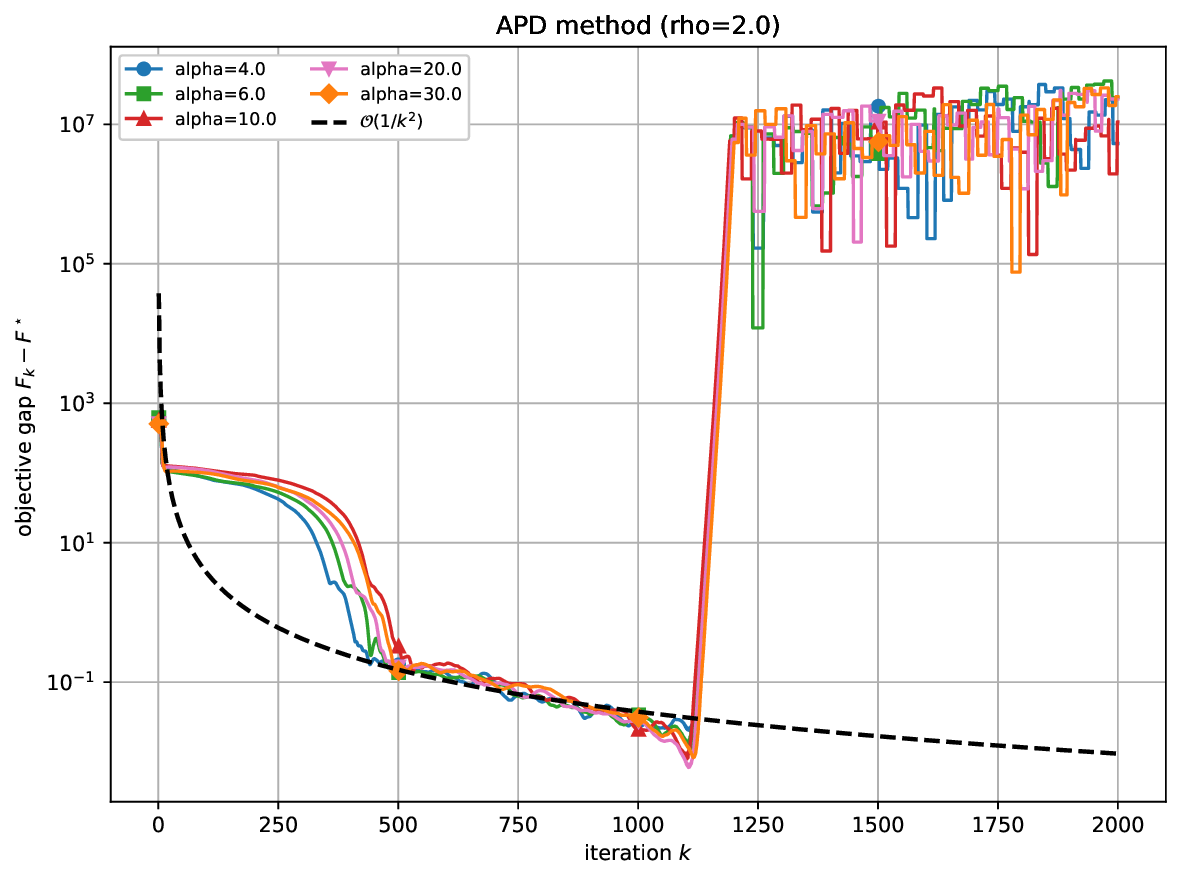}
    \caption{$\rho=2.0$}
  \end{subfigure}
  \hfill
  \begin{subfigure}{0.48\textwidth}
    \centering
    \includegraphics[width=\textwidth]{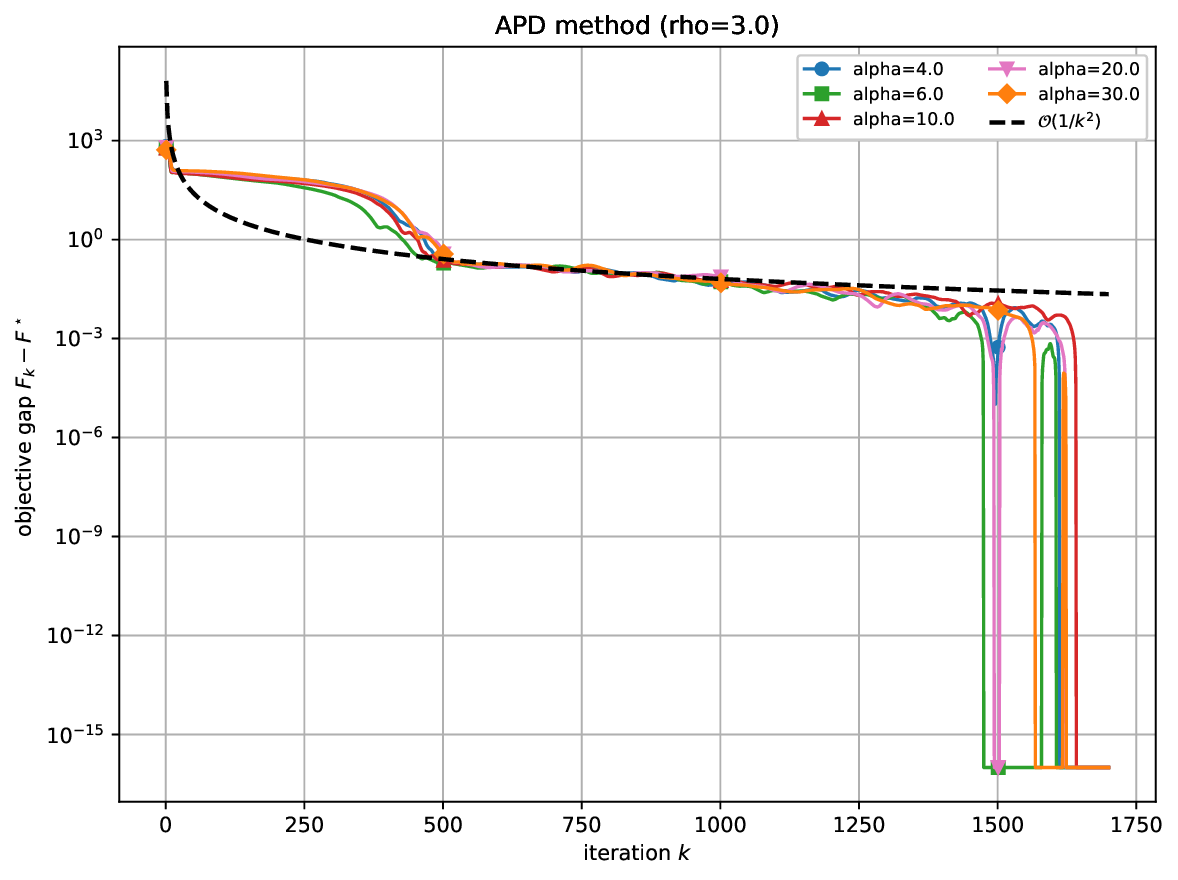}
    \caption{$\rho=3.0$}
  \end{subfigure}

  \vspace{3mm}

  \begin{subfigure}{0.48\textwidth}
    \centering
    \includegraphics[width=\textwidth]{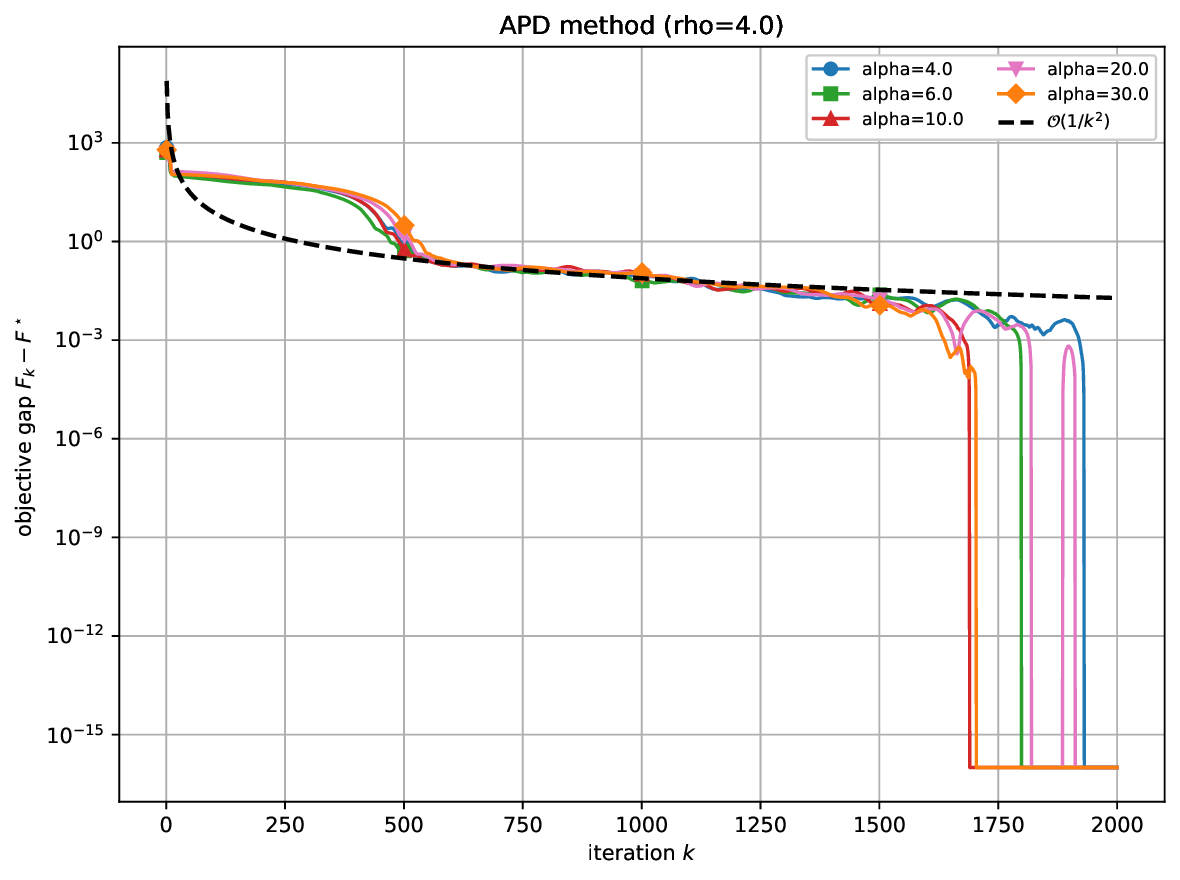}
    \caption{$\rho=4.0$}
  \end{subfigure}
  \hfill
  \begin{subfigure}{0.48\textwidth}
    \centering
    \includegraphics[width=\textwidth]{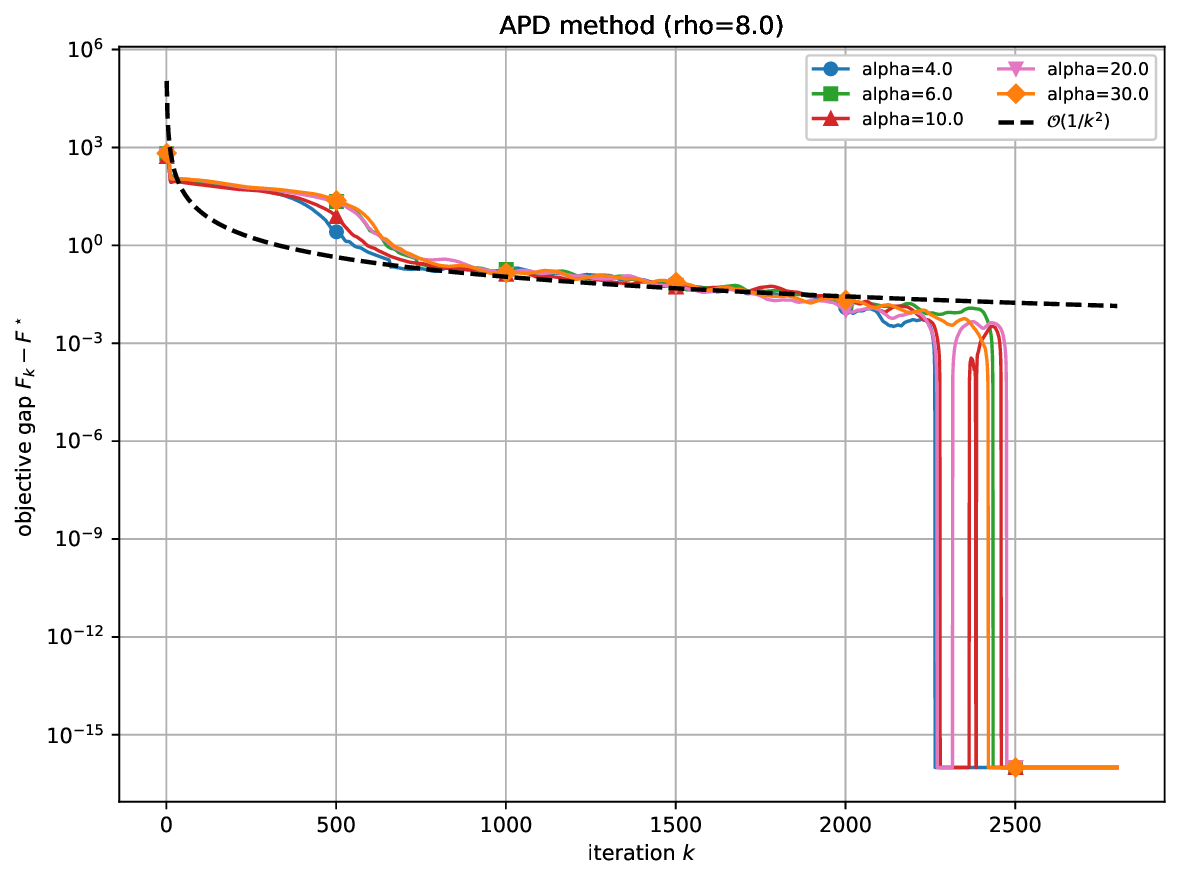}
    \caption{$\rho=8.0$}
  \end{subfigure}

  \caption{Convergence behavior for different $\alpha$ under various $\rho$.}
  \label{fig:rho_combined}
\end{figure}
\begin{figure}[htbp]
  \centering

  \begin{subfigure}{0.48\textwidth}
    \centering
    \includegraphics[width=\textwidth]{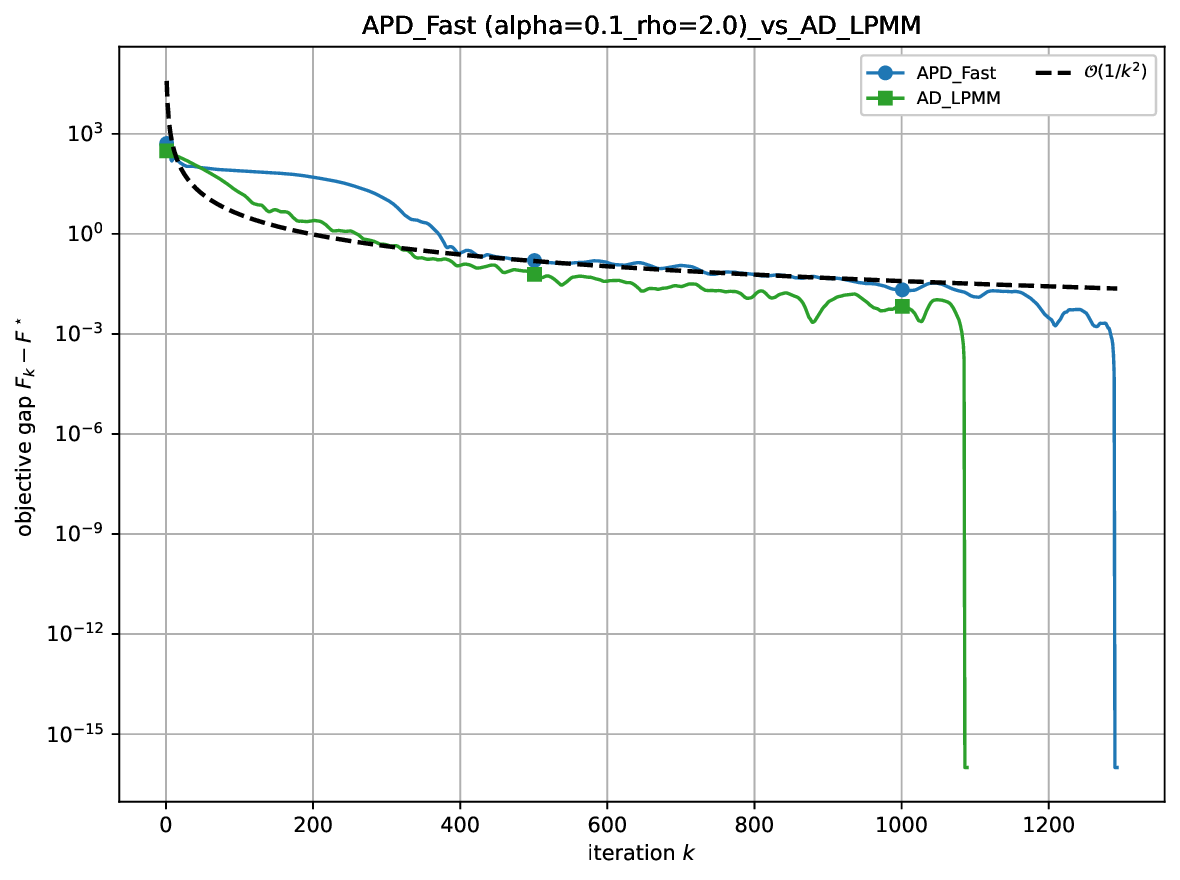}
    \caption{Proposed method with $\alpha=0.1$, $\rho=2.0$ and AD-LPMM.}
  \end{subfigure}
  \caption{Comparison between the Accelerated Primal-Dual method and the AD-LPMM.}
  \label{fig:right-group}
\end{figure}

%% file: chapter6.tex
\newpage
\section{Conclusions and discussions}
In this work, we analyzed the convergence rates of primal-dual algorithms for proper, lower semi-continuous, convex functions, 
which are not necessarily differentiable, by imposing the \textit{Flatness condition} and \textit{Growth condition} on the Lagrangian in the continuous limit. 
We also evaluated how algorithmic parameters affect convergence, providing insight for their optimal selection.
While we discretized the scheme using numerical methods and confirmed its convergence through experiments, 
a rigorous discrete-level analysis remains an open problem. 
Specifically, the lack of a formal stability analysis for the algorithm parameters $\rho$ and step size $h$ hinders effective tuning.
More fundamentally, continuous-level dynamics are not guaranteed to coincide with their discretized counterparts, and discrete schemes do not inherently equivalent to practical algorithms.
Without establishing these connections, a continuous-level analysis may lack practical significance.

Additionally, to address stochastic phenomena like mini-batch sampling in modern AI, we aim to extend our analysis to SDE-based dynamics.
 This framework enables us to view deep learning through the lens of physical dynamics and provide theoretical guarantees for modern architectures.
  Furthermore, we intend to incorporate extreme conditions and constrained formulations to handle more practical and challenging scenarios.
\section*{Acknowledgements}
This work was supported by JSPS KAKENHI Grant Numbers JP26K14720, JP23K10999 and Keio University Doctoral Student Research Support Program (Ushioda Memorial Fund).